\documentclass[11pt,a4paper,twoside]{article}
\usepackage[UKenglish]{babel}
\usepackage[T1]{fontenc}
\usepackage[utf8x]{inputenc}
\usepackage{latexsym,amsfonts,amsmath,amsthm,amssymb, mathrsfs}
\usepackage{fullpage}
\usepackage{xcolor,bm, bbm}
\usepackage{paralist}

\usepackage{fourier}

\DeclareMathOperator{\E}{\mathbb{E}}
\DeclareMathOperator{\R}{\mathbb{R}}
\DeclareMathOperator{\N}{\mathbb{N}}
\DeclareMathOperator{\prb}{\mathbb{P}}
\DeclareMathOperator{\I}{ \mathbb{I} }
\DeclareMathOperator{\U}{ \text{Unif} }

\newtheorem{thm}{Theorem}[section]

\newtheorem{proposition}[thm]{Proposition}
\newtheorem{rem}[thm]{Remark}
\newtheorem{lem}[thm]{Lemma}

\newtheorem{thmalpha}{Theorem}

\newtheorem{AssumptionA}{Assumption}

{
\theoremstyle{definition}

}

\allowdisplaybreaks
\begin{document}

\title{\bf Hölder's inequality and its reverse \\-- a probabilistic point of view }
\medskip

\author{Lorenz Fr\"uhwirth and Joscha Prochno}



\date{}

\maketitle

\begin{abstract}
In this article we take a probabilistic look at H\"older's inequality, considering the ratio of terms in the classical Hölder inequality for random vectors in $\R^n$. We proof a central limit theorem for this ratio, which then allows us to reverse the inequality up to a multiplicative constant with high probability. The models of randomness include the uniform distribution on $\ell_p^n$ balls and spheres. We also provide a Berry-Esseen type result and prove a large and a moderate deviation principle for the suitably normalized H\"older ratio. 
\end{abstract}

\section{Introduction \& Main results}

There are a number of classical inequalities frequently used throughout mathematics. A natural question is then to characterize the equality cases or to determine to what degree a reverse inequality may hold. The latter shall be the main focus here and we begin by motivating and illustrating the approach in the case of the classical arithmetic-geometric mean inequality (AGM inequality), which has attracted attention in the past decade. Let us recall that the AGM inequality states that for any finite number of non-negative real values, the geometric mean is less than or equal to the arithmetic mean. More precisely, for all $n \in \N$ and $x_1, \dots , x_n \geq 0$ it holds that
\begin{equation*}
\label{EqClassicalArithGeoIneq}
\Big( \prod_{i=1}^n x_i   \Big)^{1/n} \leq  \frac{1}{n} \sum_{i=1}^n x_i
\end{equation*}
and equality holds if and only if $x_1= \dots = x_n$. Setting $y_i := \sqrt{x_i}$ for $i=1, \dots ,n$, we obtain 
\begin{equation*}
\Big( \prod_{i=1}^n y_i   \Big)^{1/n} \leq  \Big(  \frac{1}{n} \sum_{i=1}^n y_i^2 \Big)^{1/2}.
\end{equation*}
For a point $y$ in the Euclidean unit sphere $\mathbb{S}_2^{n-1}:=\{x=(x_i)_{i=1}^n\in\R^n\,:\, \sum_{i=1}^n|x_i|^2 = 1\}$, this leads to the estimate
\begin{equation*}
\Big( \prod_{i=1}^{n} |y_i |\Big)^{1/n} \leq \frac{1}{\sqrt{n}}.
\end{equation*}
It is natural to ask whether this inequality can be reversed for a ``typical'' point in $\mathbb{S}_2^{n-1}$ and in \cite[Proposition 1]{GluMil}, Gluskin and Milman showed that for any $t \in \R$, 
\begin{equation*}
\label{EqGluMilResult}
\sigma^{(n)}_2 \Big( \Big \{   x \in \mathbb{S}_2^{n-1} \ : \ \Big( \prod_{i=1}^{n} |x_i |\Big)^{1/n}  \geq t \cdot \frac{1}{\sqrt{n}}      \Big \}    \Big) \geq 1- (1.6 \sqrt{t})^n,
\end{equation*}
where $\sigma^{(n)}_2$ denotes the unique rotationally invariant probability surface measure (the Haar measure) on $\mathbb{S}_2^{n-1}$.
For large dimensions $n \in \N$ this means that with high probability, we can reverse the AGM inequality up to a constant. The problem was then revisited by Aldaz in \cite[Theorem 2.8]{Aldaz} and he showed that for all $\epsilon > 0, k> 0$ there exists an $N:=N(k, \epsilon ) \in \N$ such that for every $n \geq N$
\begin{equation*}
\label{EqAldazResult}
\sigma^{(n)}_2\Bigg( \Bigg \{  x \in \mathbb{S}_2^{n-1} \ :\ \frac{(1- \epsilon) e^{ -\frac{1}{2} ( \gamma + \log 2)} }{\sqrt{n}}  <  \Big( \prod_{i=1}^{n} |x_i |\Big)^{1/n} < \frac{(1 + \epsilon) e^{ -\frac{1}{2} ( \gamma + \log 2)} }{\sqrt{n}}  \Bigg \}   \Bigg) \geq 1- \frac{1}{n^k},
\end{equation*}
where $\gamma=0,5772\dots$ is Euler's constant. The previous works motivated Kabluchko, Prochno, and Vysotsky \cite{ArithGeoIneq} to study the asymptotic behavior of the $p$-generalized AGM inequality, which states that for $p \in (0, \infty)$, $n \in \N$, and $(x_i)_{i=1}^n \in \R^n$, 
\begin{equation*}
\label{EqIntrpgenArithGeoIneq}
\Big( \prod_{i=1}^{n} |x_i| \Big)^{1/n} \leq \Big(   \frac{1}{n} \sum_{i=1}^{n} |x_i|^p \Big)^{1/p}.
\end{equation*} 
The authors then analyzed the quantity
\begin{equation*}
\label{EqRatioArithGeo}
\mathcal{R}_n:= \frac{\big( \prod_{i=1}^{n} |x_i| \big)^{1/n}  }{ || x||_p},
\end{equation*}
where $ || x||_p := \big( \sum_{i=1}^{n} |x_i|^p  \big)^{1/p} $, $x=(x_i)_{i=1}^n\in\R^n$. Similar as in the case of the classical AGM inequality above it is now natural to consider points $x \in \R^n$ that are uniformly distributed on the $\ell_p^n$ unit sphere $\mathbb{S}^{n-1}_p$ or the $\ell_p^n$ unit ball $\mathbb{B}^{n}_p$ respectively, where
\begin{equation*}
\mathbb{B}^n_p := \big \{ x \in \R^n \ : \ || x ||_p \leq 1 \big \} \quad \text{and} \quad \mathbb{S}^{n-1}_p := \big \{ x \in \R^n \ : \ || x ||_p = 1 \big \}.
\end{equation*} 
In \cite[Theorem 1.1]{ArithGeoIneq}, for a constant $m_p\in(0,\infty)$ only depending on $p$, it is shown that
\begin{equation*}
\sqrt{n} \big( e^{- m_p} \mathcal{R}_n -1   \big), \quad n \in \N 
\end{equation*}
converges to a centered normal distribution with known variance and in \cite[Theorem 1.3]{ArithGeoIneq} a large deviation principle for the sequence $(\mathcal{R}_n)_{n \in \N }$ is proven (see Section \ref{SecLDPProb} for the definition of an LDP).
\par{}
The work \cite{ArithGeoIneq} of Kabluchko, Prochno, and Vysotsky was then recently complemented by Th\"ale in \cite{Th2021}, who obtained a Berry-Esseen type bound and a moderate deviation principle for a wider class of distributions on the $\ell_p^n$ balls (see \cite{Bartheetal}). In the subsequent paper \cite[Theorem 1.1]{KaufThaele}, Kaufmann and Th\"ale were able to identify the sharp asymptotics of $(\mathcal{R}_n)_{n \in \N }$. 
\par{}
Another classical inequality which is used throughout mathematics and applied in numerous situations is H\"older's inequality. While, as outlined above, the AGM inequality is by now well understood from a probabilistic point of view, here we shall focus on H\"older's inequality and take a probabilistic approach in the same spirit. We recall that for $n\in\N$ and $p,q\in (1, \infty)$ with
$ \frac{1}{p} + \frac{1}{q} = 1$, H\"older's inequality states that for all points $x,y\in\R^n$,
\begin{equation}
\label{EqHoelderIneq}
\sum_{i=1}^{n} |x_i y_i | \leq \Big( \sum_{i=1}^{n} |x_i|^p \Big)^{1/p} \Big( \sum_{i=1}^{n} |y_i|^q \Big)^{1/q}.
\end{equation}
The random quantity to be analyzed is therefore the ratio 
\begin{equation}
\label{EqRationHoelderIneq}
\mathcal{R}_{p,q}^{(n)} := \frac{\sum_{i=1}^{n} |X^{(n)}_i Y^{(n)}_i |}{\Big( \sum_{i=1}^{n} |X^{(n)}_i|^p \Big)^{1/p} \Big( \sum_{i=1}^{n} |Y^{(n)}_i|^q \Big)^{1/q}}, \quad n \in \N ,
\end{equation}
where we assume that $X^{(n)}$ and $Y^{(n)}$ are independent random points in $\mathbb{B}^{n}_p$ and $ \mathbb{B}^n_q$, respectively. In fact, we focus here on the uniform distribution on $\mathbb{B}_p^n$ and $\mathbb{S}_p^{n-1}$, i.e., we consider the cases where $X^{(n)} \sim \U( \mathbb{B}^n_p)$ and  $Y^{(n)} \sim \U( \mathbb{B}^n_q)$ or $X^{(n)} \sim \U( \mathbb{S}^{n-1}_p)$ and  $Y^{(n)} \sim \U( \mathbb{S}^{n-1}_q)$. The uniform distribution on $\mathbb{B}^{n}_p$ is given by the normalized Lebesgue measure, whereas there are two meaningful uniform distributions on $\mathbb{S}^{n-1}_p $, namely the surface measure denoted by $\sigma^{(n)}_p$ and the cone probability measure $\mu^{(n)}_p$ (see Subsection \ref{SecLDPProb} for precise definitions). 

\subsection{Main results -- Limit theorems for the H\"older ratio} \label{SubSectionMainResults}

Let us now present our main results. For the sake of brevity, we first introduce the following general assumption on our random quantities.

\begin{AssumptionA}
	\label{AssA}
	Let $X^{(n)},Y^{(n)}$ be independent random vectors in $\R^n$ and let $p,q \in (1, \infty)$ with $\frac{1}{p} + \frac{1}{q} =1$. We assume either, $(X^{(n)},Y^{(n)}) \sim $ $\U(\mathbb{B}_p^{n}) \otimes \U(\mathbb{B}_q^{n})$ or $(X^{(n)},Y^{(n)}) \sim \mu^{(n)}_p \otimes \mu^{(n)}_q$ or $(X^{(n)},Y^{(n)}) \sim \sigma^{(n)}_p \otimes \sigma^{(n)}_q$, where $\mu^{(n)}_p$ and $\sigma^{(n)}_p$ denote the cone probability measure and the surface probability measure on $\mathbb{S}_p^{n-1}$, respectively.  
\end{AssumptionA}

The following quantities appear in the formulation of Theorems \ref{ThmCLT} and \ref{ThmMDP}. Let $\Gamma $ denote the Gamma-function and set
\begin{equation}
\label{EqCovMatrixVectord}
\mathbf{C_{p,q}} = \left (
\begin{matrix}
	p^{2/p} \frac{\Gamma \left( \frac{3}{p}   \right)}{\Gamma \left( \frac{1}{p}   \right)} 
	q^{2/q} \frac{\Gamma \left( \frac{3}{q}   \right)}{\Gamma \left( \frac{1}{q}   \right)} - m_{p,q}^2 & 
	m_{p,q}  &   m_{p,q}\\
	m_{p,q}&  p & 0 \\
	m_{p,q} & 0 & q
\end{matrix}
\right),
\quad d_{p,q} :=\left( 1, - \frac{m_{p,q}}{p}, - \frac{m_{p,q}}{q} \right),
\end{equation}
where $m_{p,q} := p^{1/p}        \frac{\Gamma \left( \frac{2}{p}   \right)}{\Gamma \left( \frac{1}{p}   \right)}     
q^{1/q}        \frac{\Gamma \left( \frac{2}{q}   \right)}{\Gamma \left( \frac{1}{q}   \right)}      $.

\subsubsection{The CLT and Berry-Esseen bounds for $\mathcal{R}_{p,q}^{(n)}$}

We start with the central limit theorem and a Berry-Esseen type result for the H\"older ratio $\mathcal{R}_{p,q}^{(n)}$ (see \eqref{EqRationHoelderIneq}). As a consequence we shall see that H\"older's inequality may be reversed up to a specific multiplicative constant only depending on $p$ and $q$ with high probability.

\begin{thmalpha}[Central limit theorem]
	\label{ThmCLT}
	Let $X^{(n)},Y^{(n)}$ be random vectors satisfying Assumption \ref{AssA} and let 
	$(\mathcal{R}_{p,q}^{(n)})_{n \in \N }$ be given as in \eqref{EqRationHoelderIneq}, i.e.,
	\begin{equation*}
	\mathcal{R}_{p,q}^{(n)} = \frac{\sum_{i=1}^{n} |X^{(n)}_i Y^{(n)}_i |}{\Big( \sum_{i=1}^{n} |X^{(n)}_i|^p \Big)^{1/p} \Big( \sum_{i=1}^{n} |Y^{(n)}_i|^q \Big)^{1/q}}, \quad n \in \N.
	\end{equation*}
	Then, we have 
	\begin{equation}
	\label{EqCLTRpq}
	\sqrt{n} \big( \mathcal{R}_{p,q}^{(n)} - m_{p,q} \big) \stackrel{d }{\longrightarrow} Z,
	\end{equation}
	where $Z \sim \mathcal{N}(0, \sigma_{p,q}^2)$, $\sigma_{p,q}^2 := \langle d_{p,q} , \mathbf{C_{p,q}} d_{p,q} \rangle \in (0, \infty)$ with $ \mathbf{C_{p,q}}$ and $d_{p,q}$ as in \eqref{EqCovMatrixVectord}.
\end{thmalpha}

\begin{rem}
	As a consequence of Theorem \ref{ThmCLT}, for any $t \in \R$, 
	\begin{equation*}
	\lim_{n \rightarrow \infty} \mathbb{P} \Big[ \sum_{i=1}^n |X_i^{(n)} Y_i^{(n)} | \geq \Big( \frac{t}{\sqrt{n}} + m_{p,q} \Big) ||X^{(n)}||_p ||Y^{(n)}||_q \Big] = \frac{1}{\sqrt{2 \pi } \sigma_{p,q}} \int_{t}^{\infty} e^{- \frac{x^2}{2 \sigma_{p,q}^2}} dx.
	\end{equation*}
	In particular, for $t=0$, we obtain
	\begin{equation*}
	\lim_{n \rightarrow \infty} \mathbb{P} \Big[ \sum_{i=1}^n |X_i^{(n)} Y_i^{(n)} | \geq m_{p,q} ||X^{(n)}||_p ||Y^{(n)}||_q \Big] = \frac{1}{2}.
	\end{equation*}
	This means, with a probability tending to $1/2$, we can reverse Hölder's inequality up to the explicit constant $m_{p,q} = p^{1/p}        \frac{\Gamma \left( \frac{2}{p}   \right)}{\Gamma \left( \frac{1}{p}   \right)}     
	q^{1/q}        \frac{\Gamma \left( \frac{2}{q}   \right)}{\Gamma \left( \frac{1}{q}   \right)}      $.
\end{rem}

We are also able to provide a quantitative version of Theorem \ref{ThmCLT}, i.e., a Berry-Esseen type result. For real-valued random variables $X$ and $Y $ on a common probability space, we define the Kolmogorov-distance 
\begin{equation}
\label{EqKolmDist}
d_{Kol}( X,Y) := \sup_{t \in \R} \Big| \mathbb{P} \left[ X \leq t \right] - \mathbb{P} \left[ Y \leq t \right] \Big|.
\end{equation}

\begin{thmalpha}[Berry-Esseen bound]
	\label{ThmBerryEssentype}
	Let $X^{(n)},Y^{(n)}$ be random vectors satisfying Assumption \ref{AssA} and let $(\mathcal{R}_{p,q}^{(n)})_{n \in \N }$ be given as in \eqref{EqRationHoelderIneq}. Then there exists a constant $C_{p,q} \in (0, \infty)$ only depending on $p$ and $q$, such that
	\begin{equation*}
	d_{Kol} \left (\sqrt{n} \big( \mathcal{R}_{p,q}^{(n)} - m_{p,q} \big), Z \right ) \leq C_{p,q} \frac{ \log(n)}{\sqrt{n}},
	\end{equation*}
	where $m_{p,q}= p^{1/p} \frac{\Gamma( \frac{2}{p})}{\Gamma( \frac{1}{p})}   q^{1/q} \frac{\Gamma( \frac{2}{q})}{\Gamma( \frac{1}{q})} $,  $Z \sim \mathcal{N}(0, \sigma_{p,q}^2)$ and $ \sigma_{p,q}^2 = \langle d_{p,q} , \mathbf{C_{p,q}} d_{p,q} \rangle $ is the same quantity as in Theorem \ref{ThmCLT}. 
\end{thmalpha}

\begin{rem}
	\label{RemarkThmBE1}
	Theorem \ref{ThmBerryEssentype} gives a similar asymptotic bound of the distance to a normal distribution as the classical theorem of Berry-Esseen. The $\glqq{}\log(n)\grqq$ on the right-hand side seems to be owed to our method of proof and we conjecture that this factor is not necessary.
\end{rem}

\begin{rem}
	\label{RemarkThmBE2}
	Although Theorem \ref{ThmBerryEssentype} implies Theorem \ref{ThmCLT}, we provide a direct and more self-contained proof of Theorem \ref{ThmCLT}, which also contains estimates which we use in the proof of Theorem \ref{ThmMDP}.
\end{rem}

\subsubsection{Moderate and large deviations for $\mathcal{R}_{p,q}^{(n)}$}

Two other classical types of limit theorems in probability theory are moderate and large deviations, which typically occur between the normal fluctuations scale and the larger one of a law of large numbers. Here the probabilistic behavior is indeed different and universality is replaced by a tail sensitivity, which enters rate and/or speed in a subtle way. For the definitions we refer to Section \ref{sec:notation and prelim} below.

\begin{thmalpha}[Large deviation principle]
	\label{ThmLDP}
	Let $X^{(n)},Y^{(n)}$ be random vectors satisfying Assumption \ref{AssA} and let $(\mathcal{R}_{p,q}^{(n)})_{n \in \N }$ be given as in \eqref{EqRationHoelderIneq}. Then, $ (\mathcal{R}_{p,q}^{(n)})_{n \in \N }$ satisfies a large deviation principle in $\R$ at speed $n$ and with good rate function $\mathbb{I}: \R \rightarrow [0, \infty]$ defined as
	\begin{equation}
	\label{EqGRFLDP}
	\mathbb{I}(x) := \begin{cases}
	\inf \Big \{  \Lambda^{*}( u,v,w) \ : \  x = \frac{u}{v^{1/p} w^{1/q}} \Big \}, &: x > 0 \\
	+ \infty &: x \leq 0.
	\end{cases}
	\end{equation}
	The function $\Lambda^{*} : \R^3 \rightarrow [0 , \infty ]$ is given by
	\begin{equation*}
	\Lambda^{*}( u,v,w) := \sup_{ (r,s,t) \in \R^3} \big[  su + vt+ wr - \Lambda( r,s,t) \big], \quad (u,v,w) \in \R^3,
	\end{equation*}
	where
	\begin{equation*}
	\Lambda(r,s,t) := \log \int_{\R^2} c_{p,q} \exp \Big(  r |xy| + s |x|^p +t |y|^q - \frac{|x|^p}{p} - \frac{|y|^q}{q} \Big) dx \,dy, \quad (r,s,t) \in \R^3
	\end{equation*}
	with $c_{p,q}:= \frac{1}{2 p^{1/p} \Gamma( 1 + \frac{1}{p})}  \frac{1}{2 q^{1/q} \Gamma( 1 + \frac{1}{q})}$.
\end{thmalpha}

The next result concerns the moderate deviation principle for the H\"older ratio and complements the central limit theorem and the large deviation principle already presented.

\begin{thmalpha}[Moderate deviation principle]
	\label{ThmMDP}
	Let $X^{(n)},Y^{(n)}$ be random vectors satisfying Assumption \ref{AssA} and let $(\mathcal{R}_{p,q}^{(n)})_{n \in \N }$ be given as in \eqref{EqRationHoelderIneq}. Further, assume that $(b_n)_{n \in \N}\in\R^{\mathbb N}$ is a sequence such that
	\begin{equation*}
	\lim_{n \rightarrow \infty} \frac{b_n}{\sqrt{\log n}}  = \infty \quad \text{ and } \quad \lim_{n \rightarrow \infty} \frac{b_n}{ \sqrt{n}} = 0.
	\end{equation*}
	Then $\left( \frac{\sqrt{n}}{b_n} \big( \mathcal{R}_{p,q}^{(n)} - m_{p,q} \big) \right)_{n \in \N}$ satisfies a moderate deviation principle in $\R$ at speed $(b_n^2)_{n \in \N}$ and with a good rate function $\I : \R \rightarrow [0, \infty] $ given by $ \I(t) := \frac{t^2}{2 \sigma_{p,q}^2}$, where $\sigma_{p,q}^2 = \langle d_{p,q} , \mathbf{C_{p,q}} d_{p,q} \rangle \in (0, \infty)$ with $ \mathbf{C_{p,q}}$ and $d_{p,q}$ as in \eqref{EqCovMatrixVectord}, while $m_{p,q}= p^{1/p} \frac{\Gamma( \frac{2}{p})}{\Gamma( \frac{1}{p})}   q^{1/q} \frac{\Gamma( \frac{2}{q})}{\Gamma( \frac{1}{q})} $.
\end{thmalpha}

\section{Notation and Preliminaries}\label{sec:notation and prelim}

We shall now briefly introduce the notation used throughout the text together with some background material on large deviations and some further results used in the proofs.

\subsection{Notation}

For $p \in [1, \infty)$, $d \in \N$ and $x \in \R^d$, 
\begin{equation*}
||x||_p := \left( \sum_{i=1}^{d} |x_i|^p \right)^{1/p}
\end{equation*}
denotes the $p$-norm in $\R^d$. We recall the definitions of the $\ell_p^n$ unit ball and the $\ell_p^n$ unit sphere, i.e.,
\begin{equation*}
\mathbb{B}_p^n := \left \{ x \in \R^n \ : \ ||x||_p \leq 1 \right \} \quad \text{and} \quad \mathbb{S}_p^{n-1} := \left \{ x \in \R^n \ : \ ||x||_p = 1 \right \}.
\end{equation*}
Moreover, for $x,y \in \R^d$, $ \langle x, y \rangle := \sum_{i=1}^n x_i y_i$ is the standard scalar product on $\R^d$. 
$\mathscr{B}( \R^d)$ denotes the Borel-sigma algebra on $\R^d$.
For $p \geq 1$, $\gamma_p$ denotes the $p$-generalized Gaussian distribution with Lebesgue-density
\begin{equation*}
\frac{d\gamma_p}{d x}(x) := \frac{1}{2 p^{1/p} \Gamma \left( 1 + \frac{1}{p} \right)} e^{- |x|^p/p}, \quad x \in \R. 
\end{equation*}
We denote by $\mathcal{N}( \mu, \sigma^2)$ the normal distribution with mean $\mu \in \R$ and variance $\sigma^2 \in (0, \infty)$. For two distributions $\nu_1 $ and $ \nu_2$ on $\mathscr{B}( \R^d)$, we denote by $\nu_1 \otimes \nu_2$ the product measure of $\nu_1$ and $\nu_2$. Given a sequence of real-valued random variables $(X_n)_{n \in \N}$ and another real-valued random variable $X$, we denote by $X_n \stackrel{d}{\longrightarrow} X$ convergence in distribution. We shall also write iid for independent and identically distributed.

\subsection{Basics from large deviation theory and probability}
\label{SecLDPProb}

Let $d \in \N $ and $( \xi_n)_{n \in \N}$ be a sequence of $\R^d$-valued random variables and let $(s_n)_{n \in \N}$ be a sequence of real numbers tending to infinity. We say that $( \xi_n)_{ n \in \N}$ satisfies a large deviation principle (LDP) in $\R^d$ at speed $(s_n)_{n \in \N }$ if and only if there exists a good rate function (GRF) $\I : \R^d \rightarrow [0, \infty ]$, i.e., $\I $ has compact level sets, such that
\begin{equation}
\label{EqIntrLDP}
- \inf_{ x \in A^{ \circ }} \I( x) \leq \liminf_{ n \rightarrow \infty } \frac{1}{s_n} \log \prb[ \xi_n \in A^{ \circ } ] \leq \limsup_{ n \rightarrow \infty } \frac{1}{s_n} \log \prb[ \xi_n \in \overline{A}] \leq - \inf_{ x \in  \overline{A}} \I( x) 
\end{equation} 
for all $A\in \mathscr{B}( \R^d)$.

There are different ways to show that a certain sequence of random variables satisfies an LDP, one of the most commonly used is the so called contraction principle (see, e.g., Theorem 4.2.1 in \cite{DZ2011}).

\begin{lem}[Contraction principle]
	\label{LemContractionPrinciple}
	Let $d,n\in\N$ and $f: \R^d \rightarrow \R^n$ be a continuous function. Let $( \xi_n)_{ n \in \N}$ be a sequence of random variables that satisfies an LDP in $\R^d$ at speed $(s_n)_{n \in \N }$ with GRF $\mathbb{I}: \R^d \rightarrow [0, \infty ]$. Then, the sequence $( f( \xi_n))_{n \in \N}$ satisfies an LDP in $\R^n$ at speed $( s_n)_{n \in \N}$ with the GRF $\mathbb{I}' : \R^n \rightarrow [0, \infty ]$, where
	\begin{equation*}
	\mathbb{I}'( y) := \inf_{ x \in f^{-1}( \{ y \})} \mathbb{I}(x).
	\end{equation*}
\end{lem}

We shall also use Cramér's large deviation theorem for $\R^d $-valued random variables (see, e.g., \cite[Corollary 6.1.6]{DZ2011}). 

\begin{proposition}[Cram\'er's theorem]
	\label{ThmCramer}
	Let $(X_i)_{i \in \N}$ be a sequence of iid $\R^d$-valued random variables such that $0 \in \mathcal{D}_{ \Lambda}^{ \circ }$, where
	\begin{equation*}
	\mathcal{D}_{ \Lambda} := \big \{ t \in \R^d \ : \ \Lambda(t) = \log \E [ e^{ \langle t , X_1 \rangle }] < \infty   \big \}.
	\end{equation*}
	Then, the sequence $(\xi_n)_{n \in \N}$ with 
	\begin{equation*}
	\xi_n := \frac{X_1 + \cdots + X_n }{n}, \quad n \in \N 
	\end{equation*}
	satisfies an LDP in $\R^d$ at speed $n$ with GRF $\Lambda^{*} : \R^d \rightarrow [0, \infty]$, where
	\begin{equation*}
	\Lambda^{*}(x) := \sup_{ t \in \R^d} \big[ \langle x ,t \rangle - \Lambda(t) \big].
	\end{equation*}
\end{proposition}
    
    Two sequences of $ \R^d$-valued random variables $(\xi_n)_{n \in \N}$ and $(\eta_n)_{n \in \N}$ are said to be exponentially equivalent at speed $(s_n)_{n \in \N}$, if for all $\epsilon > 0$, we have
    \begin{equation}
    \label{EqExpEquivalence}
     \lim_{n \rightarrow \infty} \frac{1}{s_n} \log  \mathbb{P} \big[ \big || \xi_n - \eta_n \big ||_2 > \epsilon  \big] = - \infty. 
    \end{equation} 
    
    The following result can be found, e.g., in \cite[Theorem 4.2.13]{DZ2011}, and states that if a sequence of random vectors satisfies an LDP and is exponentially equivalent to another sequence of random vectors, then both satisfy the same LDP.
    
    \begin{proposition}
    \label{PropExpEquivalence}
    Let $(\xi_n)_{n \in \N}$ and $(\eta_n)_{n \in \N}$ be two random $\R^d$-valued sequences. Assume that $(\xi_n)_{n \in \N}$ satisfies an LDP at speed $(s_n)_{n \in \N}$ with GRF $\I: \R^d \rightarrow [0, \infty]$. Moreover, let $(\xi_n)_{n \in \N}$ and $(\eta_n)_{n \in \N}$ be exponentially equivalent at speed $(s_n)_{n \in \N}$. Then, $(\eta_n)_{n \in \N}$ satisfies an LDP at speed $(s_n)_{n \in \N}$ with the same GRF $\I: \R^d \rightarrow [0, \infty]$. 
    \end{proposition}

   
    Let $(\xi_n)_{n \in \N}$ be a sequence of $\R^d$-valued random variables. We say that the sequence of random variables satisfies a moderate deviation principle (MDP) if and only if $\big( \frac{\xi_n}{\sqrt{n} b_n} \big)_{n \in \N} $ satisfies an LDP at speed $(b_n^2)_{n \in \N}$ and with some GRF $\mathbb{I}: \R^d \rightarrow [0, \infty]$ for some positive sequence $(b_n)_{n\in\N}$ satisfying $\lim_{n \rightarrow \infty} b_n = \infty$ and $\lim_{n \rightarrow \infty} \frac{b_n}{\sqrt{n}}=0$.
    The scaling by $\sqrt{n} b_n$ is typically faster than the scaling in a central limit theorem but slower than the scaling in a law of large numbers. This property of an MDP is nicely illustrated in the following Cramér-type theorem (see, e.g., \cite[Theorem 3.7.1]{DZ2011}).
	
	\begin{proposition}
		\label{ThmCramerMDP}
		Let $d \in \N$ and $(X_i)_{i \in \N }$ be a sequence of iid $\R^d$-valued random variables such that 
		\begin{equation*}
		\Lambda(t) = \log \E [ e^{ \langle t , X_1 \rangle }] < \infty ,
		\end{equation*}
		for all $t$ in some ball around the origin, $\E[X_1]=0$ and $\mathbf{C}$, the covariance matrix of $X_1$, is  invertible. Let $(b_n)_{n \in \N}$ be a sequence of real numbers with
		\begin{equation*}
		\lim_{n \rightarrow \infty} b_n = \infty \quad \text{ and } \quad \lim_{n \rightarrow \infty} \frac{b_n}{ \sqrt{n}} = 0.
		\end{equation*}
		Then, the sequence $(\xi_n)_{n \in \N}$ with $\xi_n := \frac{1}{b_n \sqrt{n}} \sum_{i=1}^n X_i$ satisfies an LDP in $\R$ at speed $( b_n^2)_{n \in \N}$ with GRF $\I : \R^d \rightarrow [0, \infty]$, where
		\begin{equation*}
		\I(x) :=  \frac{1}{2} \langle x, \mathbf{C}^{-1 } x \rangle , \quad x \in \R^d.
		\end{equation*}
	\end{proposition}


The following result is taken from \cite[Lemma 4.1]{APTGaussianFluct}.

\begin{proposition}
	\label{PropEstimateKolmogoroff}
	Let $Y_1,Y_2,Y_3$ be three random variables, let $Z$ be a centered Gaussian random variable with variance $ \sigma^2 \in (0, \infty)$ and let $\epsilon > 0$. Then,
	\begin{equation*}
	\sup_{t \in \R} \big | \mathbb{P}[ Y_1 + Y_2 +Y_3 \geq t] - \mathbb{P} [ Z \geq t] \big | \leq 	\sup_{t \in \R} \big | \mathbb{P}[ Y_1  \geq t] - \mathbb{P} [ Z \geq t]   \big |  +  \mathbb{P} \left [| Y_2 |  > \frac{\epsilon }{2} \right ] +  \mathbb{P} \left [| Y_3 |  > \frac{\epsilon }{2} \right ] + \frac{\epsilon }{\sqrt{2 \pi \sigma^2 }}.
	\end{equation*}
\end{proposition}

There are two meaningful uniform distributions on the $\ell_p^n$ unit sphere $\mathbb{S}_p^{n-1}$, namely the cone probability measure $\mu^{(n)}_p$ and the surface probability measure $\sigma^{(n)}_p$. In the following, we will briefly discuss their theoretical foundation as well as the relation between those two distributions. We can equip $\mathbb{S}^{n-1}_p$ with the trace Borel-sigma algebra on $\R^n$ which we denote by $\mathcal{B}( \mathbb{S}_p^{n-1})$. For $A \in \mathcal{B}(\mathbb{S}_p^{n-1})$, the cone probability measure $\mu^{(n)}_p$ is then defined as
\begin{equation*}
\mu^{(n)}_p(A):= \frac{\lambda^{(n)}([0,1]A)}{\lambda^{(n)}(\mathbb{B}_p^n)},
\end{equation*}  
where $\lambda^{(n)}$ denotes Lebesgue measure on $\mathscr{B}(\R^n)$ and $[0,1]A := \{ x \in \R^n \ : \ x= ra, \ r \in [0,1], \ a \in A\}$. By a result of Schechtman and Zinn \cite{SchechtZinn} and Rachev and Rüschendorf \cite{RachevRuesch}, we know that for $X_p^{(n)} \sim \U \left( \mathbb{B}_p^n \right) $ and $Y_p^{(n)} \sim \mu^{(n)}_p $,
\begin{align}
\begin{split}
\label{EqProbRepSchechtmannZinn}
X_p^{(n)} & \stackrel{d}{=} U^{1/n} \frac{ \zeta^{(n)} }{ || \zeta^{(n)}||_p} \\
Y_p^{(n)} & \stackrel{d}{=} \frac{ \zeta^{(n)} }{ || \zeta^{(n)}||_p},
\end{split} 
\end{align}
where $U \sim \U( [0,1])$ and $\zeta^{(n)} := ( \zeta_1, \cdots, \zeta_n )$ are independent and $( \zeta_i)_{ i \in \N}$ is an iid sequence distributed with respect to the $p$-generalized Gaussian distribution $\gamma_p$; we recall the corresponding Lebesgue-density
\begin{equation}
\label{EqDensitypGenGaussian}
\frac{d \gamma_p}{dx}(x) = \frac{1}{2 p^{1/p} \Gamma( 1 + \frac{1}{p})} e^{ - |x|^p/p}, \quad x \in \R. 
\end{equation}

Let $\sigma^{(n)}_p$ be the $(n-1)$-dimensional Hausdorff probability measure or, equivalently, the $(n-1)$-dimensional normalized Riemannian volume measure on $\mathbb{S}_p^{n-1}$, $p\in [1,\infty)$. We have the following relation between $\mu^{(n)}_p$ and $\sigma^{(n)}_p$ (see \cite[Lemma 2]{NaorRomikProjSurfMeasure}).

\begin{proposition}
Let $n \in \N$ and $1 \leq p < \infty$. Then, for all $x \in \mathbb{S}_p^{n-1}$, 
\begin{equation*}
\frac{d \sigma^{(n)}_p}{d \mu^{(n)}_p}(x) = C_{n,p} \Big( \sum_{i=1}^n |x_i|^ {2p-2} \Big)^{1/2},
\end{equation*}
where
\begin{equation*}
C_{n,p}:= \Big( \int_{\mathbb{S}_p^{n-1}} \sum_{i=1}^n |x_i|^ {2p-2} \mu^{(n)}_p(dx) \Big)^{-1/2} .	
\end{equation*} 
\end{proposition}
If $p=1,2$, it is clear that $\sigma^{(n)}_p=\mu^{(n)}_p$. We remark that in case of $p = \infty $, we know that (see \cite{NaorRomikProjSurfMeasure}) $\sigma^{(n)}_{\infty}=\mu^{(n)}_{\infty}$. In contrast, for all $p \in (1, \infty)$ with $p \neq 2$, we have that $\sigma^{(n)}_p \neq \mu^{(n)}_p$. Nevertheless, for large $n \in \N$, one can prove that $\sigma^{(n)}_p$ and $\mu^{(n)}_p$ are close in the total variation distance (see \cite[Theorem 2]{NaorRomikProjSurfMeasure}). 

\begin{proposition}
\label{PropSurfMeasConeMeasClose}
For all $1 \leq p < \infty$, we have
\begin{equation*}
|| \mu^{(n)}_p - \sigma^{(n)}_p ||_{TV} := \sup_{ A \in \mathscr{B}(\mathbb{S}_p^{n-1}) } | \mu^{(n)}_p(A) - \sigma^{(n)}_p(A) | \leq \frac{c_p}{\sqrt{n}},
\end{equation*} 
where $c_p \in (0,\infty)$ only depends on $p$.	
\end{proposition}

\section{Proofs of the main results}

Before we continue with some technical results and the proofs of the main theorems, let us recall here that all theorems stated in Section \ref{SubSectionMainResults} assume Assumption \ref{AssA}.

We begin with a technical Lemma giving a useful representation of the H\"older ratio $\mathcal{R}_{p,q}^{(n)}$, $ n \in \N$. 

\begin{lem}
\label{LemReprRpq}
Let $p,q \in (1,\infty)$ with $\frac{1}{p} + \frac{1}{q} = 1$ and assume that either, $(X^{(n)},Y^{(n)}) \sim \U( \mathbb{B}^n_p) \otimes  \U( \mathbb{B}^n_q)$ or $ (X^{(n)},Y^{(n)}) \sim \mu^{(n)}_p \otimes  \mu^{(n)}_q$. Then, we have
\begin{equation}
\label{EqRepRationRpq}
\mathcal{R}_{p,q}^{(n)} = \frac{\sum_{i=1}^{n} |X^{(n)}_i Y^{(n)}_i |}{\Big( \sum_{i=1}^{n} |X^{(n)}_i|^p \Big)^{1/p} \Big( \sum_{i=1}^{n} |Y^{(n)}_i|^q \Big)^{1/q}} \stackrel{d}{=} \frac{\sum_{i=1}^{n} |\zeta_i \eta_i |}{\Big( \sum_{i=1}^{n} |\zeta_i|^p \Big)^{1/p} \Big( \sum_{i=1}^{n} |\eta_i|^q \Big)^{1/q}},
\end{equation}
where $\big( (\zeta_i, \eta_i) \big)_{i \in \N} $ is an iid sequence with $ (\zeta_1, \eta_1) \sim \gamma_p \otimes \gamma_q $.
\end{lem}
\begin{proof}
	 First, assume that $(X^{(n)}, Y^{(n)}) \sim \U( \mathbb{B}^n_p) \otimes \U( \mathbb{B}^n_q)$. Then, using the Schechtmann-Zinn representation in \eqref{EqProbRepSchechtmannZinn}, we have $X^{(n)} \stackrel{d}{=} U^{1/n} \frac{\zeta^{(n)}}{|| \zeta^{(n)} ||_p}$, where $U$ and $\zeta^{(n)} =(\zeta_1, \cdots, \zeta_n)$ are independent with $U \sim \U([0,1])$ and iid $\zeta_i \sim \gamma_p $ for $i \in \N$. The random variable $Y^{(n)}$ has a similar form, i.e., $Y^{(n)} \stackrel{d}{=} V^{1/n} \frac{\eta^{(n)}}{|| \eta^{(n)} ||_q}$, where $V$ and $\eta^{(n)} =(\eta_1, \cdots, \eta_n)$ are independent with $V \sim \U([0,1])$ and iid $\eta_i \sim \gamma_q $ for $i \in \N$. Using this leads to
	\begin{equation*}
	\mathcal{R}_{p,q}^{(n)} = \frac{\sum_{i=1}^{n} |X^{(n)}_i Y^{(n)}_i |}{\Big( \sum_{i=1}^{n} |X^{(n)}_i|^p \Big)^{1/p} \Big( \sum_{i=1}^{n} |Y^{(n)}_i|^q \Big)^{1/q}} \stackrel{d}{=} 
	\frac{ \sum_{i=1}^{n} |\zeta_i \eta_i |}{\Big( \sum_{i=1}^{n} |\zeta_i|^p \Big)^{1/p} \Big( \sum_{i=1}^{n} |\eta_i|^q \Big)^{1/q}}.
	\end{equation*}
	Now assume that $X^{(n)}$ and $Y^{(n)}$ are independent and distributed with respect to the cone measure on $\mathbb{S}^{n-1}_p$ and $\mathbb{S}^{n-1}_q$ respectively, i.e., $(X^{(n)}, Y^{(n)}) \sim \mu_p^{(n)} \otimes \mu_q^{(n)}$. Then, again by \eqref{EqProbRepSchechtmannZinn}, we have that $X^{(n)} \stackrel{d}{=} \frac{\zeta^{(n)}}{|| \zeta^{(n)} ||_p}$ and $Y^{(n)} \stackrel{d}{=} \frac{\eta^{(n)}}{|| \eta^{(n)} ||_q}$ with $\zeta^{(n)} = (\zeta_1,...,\zeta_n) $ and $ \eta^{(n)} = (\eta_1,...,\eta_n)$, where $\big( (\zeta_i, \eta_i) \big)_{i \in \N}$ is an iid sequence with $(\zeta_1, \eta_1) \sim \gamma_p \otimes \gamma_q$. Hence, we receive the same representation in distribution for $\mathcal{R}_{p,q}^{(n)}$, i.e., we have
	\begin{equation*}
	\mathcal{R}_{p,q}^{(n)} = \frac{\sum_{i=1}^{n} |X^{(n)}_i Y^{(n)}_i |}{\Big( \sum_{i=1}^{n} |X^{(n)}_i|^p \Big)^{1/p} \Big( \sum_{i=1}^{n} |Y^{(n)}_i|^q \Big)^{1/q}} \stackrel{d}{=} \frac{\sum_{i=1}^{n} |\zeta_i \eta_i |}{\Big( \sum_{i=1}^{n} |\zeta_i|^p \Big)^{1/p} \Big( \sum_{i=1}^{n} |\eta_i|^q \Big)^{1/q}}, \quad n \in \N.
	\end{equation*}
\end{proof}

\subsection{Proof of Theorem \ref{ThmCLT}} 

We start with an auxiliary Lemma that is used in the proof of the central limit theorem stated as Theorem \ref{ThmCLT} and also later in the proof of Theorem \ref{ThmMDP} (see Section \ref{SectionProofMDP}). 

\begin{lem}
\label{LemTaylorCLT}
Let $X^{(n)}$ and $Y^{(n)}$ be two independent random vectors and assume that either, $(X^{(n)}, Y^{(n)}) \sim \U(\mathbb{B}_p^{n-1}) \otimes \U(\mathbb{B}_q^{n-1})$ or $(X^{(n)}, Y^{(n)}) \sim \mu^{(n)}_p \otimes \mu^{(n)}_q$ and let
\begin{equation*}
 \mathcal{R}_{p ,q}^{(n)} = \frac{\sum_{i=1}^{n} |X^{(n)}_i Y^{(n)}_i |}{\Big( \sum_{i=1}^{n} |X^{(n)}_i|^p \Big)^{1/p} \Big( \sum_{i=1}^{n} |Y^{(n)}_i|^q \Big)^{1/q}}.
\end{equation*} 
Then, we can write
\begin{equation}
\label{EqTaylorApproxRpq}
\mathcal{R}_{p ,q}^{(n)} = m_{p,q} + \frac{1}{\sqrt{n}} S_n^{(1)} - \frac{m_{p,q}}{p \sqrt{n}} S_n^{(2)} -  \frac{m_{p,q}}{q \sqrt{n}} S_n^{(3)} + R \Big(  \frac{S_n^{(1)}}{\sqrt{n}} , \frac{S_n^{(2)}}{\sqrt{n}}, \frac{S_n^{(3)}}{\sqrt{n}} \Big),
\end{equation}
where $m_{p,q} = p^{1/p} \frac{\Gamma( \frac{2}{p})}{\Gamma( \frac{1}{p})}   q^{1/q} \frac{\Gamma( \frac{2}{q})}{\Gamma( \frac{1}{q})} $, $ S_n^{(1)} := \frac{1}{\sqrt{n}} \sum_{i=1}^n ( | \zeta_i \eta_i | - m_{p , q}) $, $S_n^{(2)} := \frac{1}{\sqrt{n}} \sum_{i=1}^n ( | \zeta_i |^p -  1) $ and $S_n^{(3)} := \frac{1}{\sqrt{n}} \sum_{i=1}^n ( | \eta_i |^q -  1)$ with iid $(\zeta_i , \eta_i ) \sim \gamma_p \otimes \gamma_q$, $i \in \N$. The function $R$ has the property that there is an $M \in ( 0, \infty ) $ such that 
\begin{equation}
\label{EqErrorEstimate}
|R(x,y,z)| \leq M || (x,y,z) ||^2_2, \quad \text{ as  } \quad || (x,y,z) ||_2 \rightarrow 0. 
\end{equation}
\end{lem}

\begin{proof}(of Lemma \ref{LemTaylorCLT})
	\label{ProofTheoremCLT}
	Consider the function $F: D_F \rightarrow \R$ with
	\begin{equation*}
	F(x,y,z) := \frac{x + m_{p,q}}{(1+y)^{1/p} ( 1+ z)^{1/q}},
	\end{equation*}
	where $D_F \subseteq \R^3$ is the domain of $F$ and $m_{p,q}$ is the constant from Lemma \ref{LemTaylorCLT}. Clearly, $F$ is twice continuously differentiable in $D_F$ which contains an open neighborhood of $(0,0,0)$. So, the Taylor expansion of first order exists locally around $(0,0,0)$ and, for $(x,y,z) \in D_F$, we get
	\begin{equation}
	\label{EqTaylorExpF}
	F(x,y,z) = m_{p,q} + x - \frac{m_{p,q}}{p} y - \frac{m_{p,q}}{q} z + R(x,y,z),
	\end{equation}
	where there exists $M, \delta  \in (0, \infty)$ such that, for $||(x,y,z)||_2 \leq \delta $, we have $ |R(x,y,z)| \leq M ||(x,y,z)||_2 ^2$. 
	Using the representation of $\mathcal{R}_{p,q}^{(n)}$ from Lemma \ref{LemReprRpq}, it follows that
	\begin{align*}
	\mathcal{R}_{p,q}^{(n)} & \stackrel{d}{=}  \frac{\sum_{i=1}^{n} |\zeta_i \eta_i |}{\Big( \sum_{i=1}^{n} |\zeta_i|^p \Big)^{1/p} \Big( \sum_{i=1}^{n} |\eta_i|^q \Big)^{1/q}} \\
	& = 
	 \frac{\frac{1}{n} \sum_{i=1}^{n} \big( |\zeta_i \eta_i | - m_{p,q} \big) + m_{p,q} }{\Big( \frac{1}{n}  \sum_{i=1}^{n} \big( |\zeta_i|^p -1 \big)  +1\Big)^{1/p} \Big( \frac{1}{n}  \sum_{i=1}^{n} \big( |\eta_i|^q -1 \big)  +1\Big)^{1/q}} \\
	 & = \frac{ \frac{1}{\sqrt{n}}S_n^{(1)} + m_{p,q} }{\big(  \frac{1}{\sqrt{n}}S_n^{(2)} +1 \big)^{1/p} \big( \frac{1}{\sqrt{n}}S_n^{(3)}  +1 \big)^{1/q}} \\
	 & = F \left(\frac{S_n^{(1)}}{\sqrt{n}},\frac{S_n^{(2)}}{\sqrt{n}},\frac{S_n^{(3)}}{\sqrt{n}} \right),
	\end{align*}
	where we have used that $ \frac{1}{p} + \frac{1}{q} = 1$ and the quantities $S_n^{(i)}, i =1,2,3$ are given as in Lemma \ref{LemTaylorCLT}. By the Taylor expansion of $F$ in \eqref{EqTaylorExpF}, we get
	\begin{equation*}
	\mathcal{R}_{p ,q}^{(n)} = m_{p,q} + \frac{1}{\sqrt{n}} S_n^{(1)} - \frac{m_{p,q}}{p \sqrt{n}} S_n^{(2)} -  \frac{m_{p,q}}{q \sqrt{n}} S_n^{(3)} + R \Big(  \frac{S_n^{(1)}}{\sqrt{n}} , \frac{S_n^{(2)}}{\sqrt{n}}, \frac{S_n^{(3)}}{\sqrt{n}} \Big),
	\end{equation*}
	as claimed.
\end{proof}
\begin{proof}[Proof of Theorem \ref{ThmCLT}]
First, we assume that either, $(X^{(n)}, Y^{(n)}) \sim \U(\mathbb{B}_p^{n-1}) \otimes \U(\mathbb{B}_q^{n-1})$ or $(X^{(n)}, Y^{(n)}) \sim \mu^{(n)}_p \otimes \mu^{(n)}_q$. Then, we can use the Taylor expansion in \eqref{EqTaylorApproxRpq} from Lemma \ref{LemTaylorCLT}, where we get
\begin{equation}
\label{EqRepCLTRpq}
\sqrt{n} \big(  \mathcal{R}_{p,q}^{(n)} - m_{p,q} \Big) \stackrel{d}{=} S_n^{(1)} - \frac{m_{p,q}}{p} S_n^{(2)} - \frac{m_{p,q}}{q} S_n^{(3)} + \sqrt{n} 
R \Big(  \frac{S_n^{(1)}}{\sqrt{n}} , \frac{S_n^{(2)}}{\sqrt{n}}, \frac{S_n^{(3)}}{\sqrt{n}} \Big),
\end{equation}
where $ S_n^{(1)} := \frac{1}{\sqrt{n}} \sum_{i=1}^n ( | \zeta_i \eta_i | - m_{p , q}) $, $S_n^{(2)} := \frac{1}{\sqrt{n}} \sum_{i=1}^n ( | \zeta_i |^p -  1) $ and $S_n^{(3)} := \frac{1}{\sqrt{n}} \sum_{i=1}^n ( | \eta_i |^q -  1)$ with iid $(\zeta_i , \eta_i ) \sim \gamma_p \otimes \gamma_q$, $i \in \N$. We show that 
\begin{equation}
\label{EqConvinProbR}
\sqrt{n} R \Big(  \frac{S_n^{(1)}}{\sqrt{n}} , \frac{S_n^{(2)}}{\sqrt{n}}, \frac{S_n^{(3)}}{\sqrt{n}} \Big) \stackrel{\mathbb{P}}{\longrightarrow} 0.
\end{equation}
To that end, we recall \eqref{EqErrorEstimate} from Lemma \ref{LemTaylorCLT}. There, we showed that for sufficiently small $\delta > 0$ and $(x,y,z) \in \R^3$ with $ || (x,y,z) ||_2 < \delta $, we have that $ | R(x,y,z) | \leq M ||(x,y,z)||_2^2 $ for some constant $M \in (0, \infty)$. This gives the following estimate
\begin{align*}
\mathbb{P} \left[ \frac{ R \Big(  \frac{S_n^{(1)}}{\sqrt{n}} , \frac{S_n^{(2)}}{\sqrt{n}}, \frac{S_n^{(3)}}{\sqrt{n}} \Big)}{ \Big | \Big |    \Big(  \frac{S_n^{(1)}}{\sqrt{n}} , \frac{S_n^{(2)}}{\sqrt{n}}, \frac{S_n^{(3)}}{\sqrt{n}} \Big) \Big | \Big |_2^2 }  > M \right] & \leq \mathbb{P} \Big[  \Big | \Big |    \Big(  \frac{S_n^{(1)}}{\sqrt{n}} , \frac{S_n^{(2)}}{\sqrt{n}}, \frac{S_n^{(3)}}{\sqrt{n}} \Big) \Big | \Big |_2^2  > \delta^2  \Big] \\
& = \mathbb{P} \Big[      \Big(  \frac{S_n^{(1)}}{\sqrt{n}} \Big)^2 + \Big(  \frac{S_n^{(2)}}{\sqrt{n}} \Big)^2 + \Big(  \frac{S_n^{(3)}}{\sqrt{n}} \Big)^2   > \delta^2  \Big] \longrightarrow 0, \quad \text{ as } \quad n \rightarrow \infty. 
\end{align*}
The latter holds due to Slutsky's theorem and the fact that $ \frac{S_n^{(i)}}{\sqrt{n}} \stackrel{\mathbb{P}}{\longrightarrow}0$, $i =1,2,3$ as $n \rightarrow \infty $ by the strong law of large numbers (note that $ \mathbb{E} [ | \zeta_1| |\eta_1|] = m_{p,q}$ and $\mathbb{E} [ | \zeta_1|^p] = \mathbb{E} [ |\eta_1|^q]=1$ ). Further, we have for $\epsilon > 0$,
\begin{align*}
\mathbb{P} \left[ \sqrt{n} R \Big(  \frac{S_n^{(1)}}{\sqrt{n}} , \frac{S_n^{(2)}}{\sqrt{n}}, \frac{S_n^{(3)}}{\sqrt{n}} \Big) > \epsilon \right] & \leq \mathbb{P} \left[ \sqrt{n}  \Big |  \Big |\Big(  \frac{S_n^{(1)}}{\sqrt{n}} , \frac{S_n^{(2)}}{\sqrt{n}}, \frac{S_n^{(3)}}{\sqrt{n}} \Big) \Big |  \Big |_2^2  > \frac{\epsilon }{M}\right ] + \mathbb{P} \left[ \frac{ R \Big(  \frac{S_n^{(1)}}{\sqrt{n}} , \frac{S_n^{(2)}}{\sqrt{n}}, \frac{S_n^{(3)}}{\sqrt{n}} \Big)}{ \Big | \Big |    \Big(  \frac{S_n^{(1)}}{\sqrt{n}} , \frac{S_n^{(2)}}{\sqrt{n}}, \frac{S_n^{(3)}}{\sqrt{n}} \Big) \Big | \Big |_2^2 }  > M \right] \\
& = \mathbb{P} \Big[     \frac{ \big( S_n^{(1)} \big)^2}{\sqrt{n}} + \frac{ \big( S_n^{(2)} \big)^2}{\sqrt{n}} +    \frac{ \big( S_n^{(3)} \big)^2}{\sqrt{n}}    > \frac{\epsilon}{M} \Big] 
+
 \mathbb{P} \left[ \frac{ R \Big(  \frac{S_n^{(1)}}{\sqrt{n}} , \frac{S_n^{(2)}}{\sqrt{n}}, \frac{S_n^{(3)}}{\sqrt{n}} \Big)}{ \Big | \Big |    \Big(  \frac{S_n^{(1)}}{\sqrt{n}} , \frac{S_n^{(2)}}{\sqrt{n}}, \frac{S_n^{(3)}}{\sqrt{n}} \Big) \Big | \Big |_2^2 }  > M \right].
\end{align*}
The second term tends to zero as $n \rightarrow \infty$ as shown before. For the first term, we observe that $ S_n^{(i)}$ converges to a normal distribution as $n \rightarrow \infty$ for $i=1,2,3$. Thus, $ \frac{(S_n^{(i)})^2}{\sqrt{n}} \stackrel{\mathbb{P}}{\longrightarrow} 0$ as $n \rightarrow \infty$ and hence, again employing Slutsky's theorem, we get
\begin{equation*}
\mathbb{P} \Big[     \frac{ \big( S_n^{(1)} \big)^2}{\sqrt{n}} + \frac{ \big( S_n^{(2)} \big)^2}{\sqrt{n}} +    \frac{ \big( S_n^{(3)} \big)^2}{\sqrt{n}}    > \frac{\epsilon}{M} \Big] \longrightarrow 0, \quad \text{as} \quad n \rightarrow \infty.
\end{equation*}
This completes the argument and shows the claim in \eqref{EqConvinProbR}. 
Now, let us consider the sequence
\begin{equation}
\label{EqVecofSis}
S_n^{(1)} - \frac{m_{p,q}}{p} S_n^{(2)} - \frac{m_{p,q}}{q} S_n^{(3)} = \frac{1}{\sqrt{n}} \sum_{i=1}^n \left( | \zeta_i \eta_i| - m_{p,q} - \frac{m_{p,q}}{p} \left( | \zeta_i|^p -1 \right) - \frac{m_{p,q}}{q} \left(|\eta_i|^q -1 \right)  \right), \quad n \in \N .
\end{equation}
We observe that \eqref{EqVecofSis} is a sum of iid scaled and centered random variables with finite second moment. Thus, by the central limit theorem, \eqref{EqVecofSis} converges in distribution to a centered normal distribution with variance
\begin{align}
\label{EqVarianceCLT}
\sigma_{p,q}^2 :&= \mathbb{V} \left[ | \zeta_1 \eta_1 | - \frac{m_{p,q}}{p} | \zeta_1 |^p - \frac{m_{p,q}}{q} | \eta_1|^q \right] \\
\notag
&= 
\mathbb{V} \left [  \left \langle d_{p,q} , \left( | \zeta_1 \eta_1 | , | \zeta_1|^p, | \eta_1 |^q   \right)  \right \rangle  \right] \\
\notag 
& = 
 \left \langle d_{p,q} , \mathbf{C_{p,q}} d_{p,q} \right \rangle.
\end{align}
Where $ \mathbf{C_{p,q}}$ is the covariance matrix of the vector $\left( | \zeta_1 \eta_1 | , | \zeta_1|^p, | \eta_1 |^q   \right)$ and $d_{p,q} = \left(  1, - \frac{m_{p,q}}{p},  - \frac{m_{p,q}}{q}   \right) $. 
We note that the vector $ \left( | \zeta_1 \eta_1 | , | \zeta_1 |^p, | \eta_1 |^q   \right) \in \R^3$ has linear independent coordinates. Thus, the covariance matrix $ \mathbf{C_{p,q}}$ is positive definite. Moreover, since $( \zeta_1 , \eta_1) \sim \gamma_p \otimes \gamma_q$, we can compute the entries of $\mathbf{C_{p,q}}$ explicitly, where we get
\begin{equation}
\label{EqCovMatrixGammapq}
\mathbf{C_{p,q}} = \left (
\begin{matrix}
p^{2/p} \frac{\Gamma \left( \frac{3}{p}   \right)}{\Gamma \left( \frac{1}{p}   \right)} 
q^{2/q} \frac{\Gamma \left( \frac{3}{q}   \right)}{\Gamma \left( \frac{1}{q}   \right)} - m_{p,q}^2 & 
m_{p,q}  &   m_{p,q}\\
m_{p,q}&  p & 0 \\
 m_{p,q} & 0 & q
\end{matrix}
\right),
\end{equation}
with $m_{p,q} = p^{1/p}        \frac{\Gamma \left( \frac{2}{p}   \right)}{\Gamma \left( \frac{1}{p}   \right)}     
q^{1/q}        \frac{\Gamma \left( \frac{2}{q}   \right)}{\Gamma \left( \frac{1}{q}   \right)}      $.
This shows that $ \sigma_{p,q}^2 $ is positive and finite. By Slutsky's theorem we hence, as $n \rightarrow \infty$, we get the following limit in distribution  claimed in Theorem \ref{ThmCLT},
\begin{equation*}
\sqrt{n} \big(  \mathcal{R}_{p,q}^{(n)} - m_{p,q} \Big) \stackrel{d}{=} S_n^{(1)} - \frac{m_{p,q}}{p} S_n^{(2)} - \frac{m_{p,q}}{q} S_n^{(3)} + \sqrt{n} 
R \Big(  \frac{S_n^{(1)}}{\sqrt{n}} , \frac{S_n^{(2)}}{\sqrt{n}}, \frac{S_n^{(3)}}{\sqrt{n}} \Big) \stackrel{d}{\longrightarrow} \mathcal{N}(0, \sigma_{p,q}^2).
\end{equation*}
\par{}
Now we consider the case when $(X^{(n)}, Y^{(n)}) \sim \sigma^{(n)}_p \otimes \sigma^{(n)}_q, n \in \N$. Let $A \in \mathcal{B}(\R)$ and recall the set $ D_n = \Big \{  (x,y) \in \mathbb{S}_p^{n-1} \times \mathbb{S}_q^{n-1} \ : \ \frac{ \sum_{i=1}^n |x_i y_i |}{|| x ||_p ||y||_q} \in A \Big \}$. Then, we have $\mathbb{P} \Big[ \mathcal{R}_{p,q}^{(n)} \in A \Big] = \sigma^{(n)}_p \otimes \sigma^{(n)}_q (D_n)$. Let $Z \sim \mathcal{N}(0, \sigma_{p,q}^2)$. Then 
\begin{equation*}
\Big| \mathbb{P} \big [  \mathcal{R}_{p,q}^{(n)} \in A \big] - \mathbb{P} \big [Z \in A \big ] \Big | \leq
\Big| \sigma^{(n)}_p \otimes \sigma^{(n)}_q \big [  D_n \big] - \mu^{(n)}_p \otimes \mu^{(n)}_q \big [ D_n \big ] \Big | 
+ \Big| \mu^{(n)}_p \otimes \mu^{(n)}_q \big [ D_n \big ] - \mathbb{P} \big [Z \in A \big ] \Big | \stackrel{n \rightarrow \infty}{\longrightarrow} 0.
\end{equation*}
The second term in the previous expression tends to zero as seen in the first part of this proof. The first term tends to zero, since
\begin{align*}
\Big| \sigma^{(n)}_p \otimes \sigma^{(n)}_q \big [  D_n \big] - \mu^{(n)}_p \otimes \mu^{(n)}_q \big [ D_n \big ] \Big | & \leq \sup_{A \in \mathcal{B}( \mathbb{S}_p^{n-1}), B \in  \mathcal{B}( \mathbb{S}_q^{n-1}) } \left | \sigma^{(n)}_p \otimes \sigma^{(n)}_q \big [  A \times B \big] - \mu^{(n)}_p \otimes \mu^{(n)}_q \big [ A \times B  \big ] \right| \\
& \leq \sup_{A \in \mathcal{B}( \mathbb{S}_p^{n-1}), B \in  \mathcal{B}( \mathbb{S}_q^{n-1}) } \left | \sigma^{(n)}_p \otimes \sigma^{(n)}_q \big [  A \times B \big] - \sigma_p^{(n)}\big [A \big ] \mu_q^{(n)} \big [B \big ]   \right| \\
& \qquad + \sup_{A \in \mathcal{B}( \mathbb{S}_p^{n-1}), B \in  \mathcal{B}( \mathbb{S}_q^{n-1}) } \left | \sigma^{(n)}_p \big [A \big ] \mu^{(n)}_q \big [B \big]  -  \mu^{(n)}_p \otimes \mu^{(n)}_q \big [ A \times B  \big ]  \right| \\
& \leq 
\sup_{ B \in \mathcal{B}( \mathbb{S}_q^{n-1}) } \left | \sigma^{(n)}_q  \big [   B \big ]-  \mu_q^{(n)} \big [B \big ]  \right| 
+
\sup_{ A \in \mathcal{B}( \mathbb{S}_p^{n-1}) } \left | \sigma^{(n)}_p \big [  A \big ]-  \mu_p^{(n)} \big [ A   \big ] \right| \stackrel{n \rightarrow \infty}{\longrightarrow} 0. 
\end{align*}
The latter follows from Proposition \ref{PropSurfMeasConeMeasClose}. Further, we used that $D_n = D_n^1 \times D_n^2$, where $D_n^1 \in \mathscr{B}(\mathbb{S}_p^{n-1})$ and $D_n^2 \in \mathscr{B}(\mathbb{S}_q^{n-1})$ (since $D_n \in \mathscr{B}(\mathbb{S}_p^{n-1} \times \mathbb{S}_q^{n-1}) = \mathscr{B}(\mathbb{S}_p^{n-1}) \times \mathscr{B}(\mathbb{S}_q^{n-1}) $, where the latter holds by, e.g. \cite[Theorem D.4]{DZ2011}).
\end{proof}

\subsection{Proof of Theorem \ref{ThmBerryEssentype}}

We now present the proof of the Berry-Esseen bound.

\begin{proof}[Proof of Theorem \ref{ThmBerryEssentype}]
First, we assume that either, $(X^{(n)}, Y^{(n)}) \sim \U(\mathbb{B}_p^{n}) \otimes \U(\mathbb{B}_q^{n})$ or $(X^{(n)}, Y^{(n)}) \sim \mu^{(n)}_p \otimes \mu^{(n)}_q$. Then, we recall identity \eqref{EqTaylorApproxRpq} from Lemma \ref{LemTaylorCLT}, i.e.,
\begin{equation*}
\sqrt{n} \big(  \mathcal{R}_{p,q}^{(n)} - m_{p,q} \Big) \stackrel{d}{=} S_n^{(1)} - \frac{m_{p,q}}{p} S_n^{(2)} - \frac{m_{p,q}}{q} S_n^{(3)} + \sqrt{n} 
R \Big(  \frac{S_n^{(1)}}{\sqrt{n}} , \frac{S_n^{(2)}}{\sqrt{n}}, \frac{S_n^{(3)}}{\sqrt{n}} \Big),
\end{equation*}
where $ S_n^{(1)} := \frac{1}{\sqrt{n}} \sum_{i=1}^n ( | \zeta_i \eta_i | - m_{p , q}) $, $S_n^{(2)} := \frac{1}{\sqrt{n}} \sum_{i=1}^n ( | \zeta_i |^p -  1) $ and $S_n^{(3)} := \frac{1}{\sqrt{n}} \sum_{i=1}^n ( | \eta_i |^q -  1)$ with iid $(\zeta_i , \eta_i ) \sim \gamma_p \otimes \gamma_q$, $i \in \N$. We can apply Proposition \ref{PropEstimateKolmogoroff} with $Y_1 := S_n^{(1)} - \frac{m_{p,q}}{p} S_n^{(2)} - \frac{m_{p,q}}{q} S_n^{(3)} $, $Y_2 := \sqrt{n} 
R \Big(  \frac{S_n^{(1)}}{\sqrt{n}} , \frac{S_n^{(2)}}{\sqrt{n}}, \frac{S_n^{(3)}}{\sqrt{n}} \Big)$ and $Y_3 := 0$, which yields, for $Z \sim  \mathcal{N}(0 , \sigma_{p,q}^2)$ and $\epsilon > 0$, 
\begin{align}
\begin{split}
\label{EqKolDistRZ}
d_{Kol} \left( \sqrt{n} \big(  \mathcal{R}_{p,q}^{(n)} - m_{p,q} \Big), Z   \right) & \leq d_{Kol} \left(   S_n^{(1)} -  \frac{m_{p,q}}{p} S_n^{(2)} - \frac{m_{p,q}}{q} S_n^{(3)} , Z \right)  \\
& 
+
\mathbb{P} \left[  \sqrt{n} \left | R \left(  \frac{S_n^{(1)}}{\sqrt{n}} , \frac{S_n^{(2)}}{\sqrt{n}}, \frac{S_n^{(3)}}{\sqrt{n}} \right )  \right | > \frac{\epsilon}{2}  \right] + \frac{\epsilon }{\sqrt{2 \pi \sigma_{p,q}^2 }}.
\end{split}
\end{align}
By the definition of $(S_n^{(i)})_{n \in \N}$ for $i=1,2,3$, we have 
\begin{equation*}
S_n^{(1)} -  \frac{m_{p,q}}{p} S_n^{(2)} - \frac{m_{p,q}}{q} S_n^{(3)} = \frac{1}{\sqrt{n}} \sum_{i=1}^n \left(  | \zeta_i \eta_i | -  m_{p,q} + \frac{m_{p,q}}{p} \left( | \zeta_i |^p  - 1 \right) + \frac{m_{p,q}}{q} \left(  | \eta_i |^q   - 1 \right) \right),
\end{equation*}
which is a sum of iid centered random variables with finite third moments. Hence, the classical Berry-Esseen theorem (see, e.g., \cite[Chapter XVI.5, Theorem 1]{Feller1971}) gives us a constant $C_1 \in (0, \infty)$ such that
\begin{equation}
\label{EqBerryEssenTaylorPoly}
d_{Kol} \left(   S_n^{(1)} -  \frac{m_{p,q}}{p} S_n^{(2)} - \frac{m_{p,q}}{q} S_n^{(3)} , Z \right)
\leq \frac{C_1}{\sqrt{n}}, \quad n \in \N. 
\end{equation}
Now, we establish an upper bound of the same order for
\begin{equation*}
\mathbb{P} \left[  \sqrt{n} \left | R \left(  \frac{S_n^{(1)}}{\sqrt{n}} , \frac{S_n^{(2)}}{\sqrt{n}}, \frac{S_n^{(3)}}{\sqrt{n}} \right)  \right |> \frac{\epsilon}{2}  \right], \quad n \in \N. 
\end{equation*}
We recall the local behavior of the function $R$ around zero given in Lemma \ref{LemTaylorCLT}. We have that there exist constants $M , \delta \in (0, \infty)$ such that $ | R( x,y,z) | \leq M || (x,y,z) ||_2^2$ for all $(x,y,z) \in \R^3$ with $||(x,y,z) ||_2 \leq \delta$. This gives us the following estimate
\begin{align}
\begin{split}
\label{EqErrorTermBE} 
\mathbb{P} \left[  \sqrt{n} \left | R \left(  \frac{S_n^{(1)}}{\sqrt{n}} , \frac{S_n^{(2)}}{\sqrt{n}}, \frac{S_n^{(3)}}{\sqrt{n}} \right) \right | > \frac{\epsilon}{2}  \right]  & \leq \mathbb{P} \left[  \left | \left |    \left(  \frac{S_n^{(1)}}{\sqrt{n}},  \frac{S_n^{(2)}}{\sqrt{n}},  \frac{S_n^{(3)}}{\sqrt{n}} \right) \right | \right |_2 >  \sqrt{ \frac{\epsilon}{2 \sqrt{n} M }}   \right] \\
& + 
\mathbb{P} \left[  \left | \left |    \left(  \frac{S_n^{(1)}}{\sqrt{n}} , \frac{S_n^{(2)}}{\sqrt{n}}, \frac{S_n^{(3)}}{\sqrt{n}} \right) \right | \right |_2 > \delta   \right]. 
\end{split}
\end{align}
For a random vector $Y =( Y_1,Y_2,Y_3) \in \R^3$ and for every $ \overline{\delta} \in (0, \infty ) $, we have the following upper bound
\begin{equation*}
\mathbb{P} \left[  || Y ||_2 >   \overline{\delta} \right] \leq \mathbb{P} \left[   |Y_1| >   \frac{\overline{\delta}}{\sqrt{3}} \right] + \mathbb{P} \left[   |Y_2| >   \frac{\overline{\delta}}{\sqrt{3}} \right]
+
\mathbb{P} \left[   |Y_3| >   \frac{\overline{\delta}}{\sqrt{3}} \right].
\end{equation*}
Applying this to the right-hand side in Equation \eqref{EqErrorTermBE} leads to
\begin{align*}
\mathbb{P} \left[  \sqrt{n} \left | R \left(  \frac{S_n^{(1)}}{\sqrt{n}} , \frac{S_n^{(2)}}{\sqrt{n}}, \frac{S_n^{(3)}}{\sqrt{n}} \right) \right | > \frac{\epsilon}{2}  \right] 
\leq P_n \left(  \frac{S_n^{(1)}}{\sqrt{n}} \right) + P_n \left(  \frac{S_n^{(2)}}{\sqrt{n}} \right) + P_n \left(  \frac{S_n^{(3)}}{\sqrt{n}} \right),
\end{align*}
where $P_n \left (\frac{S_n^{(i)}}{\sqrt{n}} \right ) := \mathbb{P} \left[   \left | \frac{S_n^{(i)}}{\sqrt{n}}    \right | > \sqrt{ \frac{\epsilon}{6 \sqrt{n} M }} \right] + \mathbb{P} \left[   \left | \frac{S_n^{(i)}}{\sqrt{n}}    \right | >  \frac{\delta}{\sqrt{3}} \right]$, $i=1,2,3$. To bound these quantities, we use \cite[Lemma 2.9]{JPBerryEsseenlp} with $\epsilon = \epsilon_n = \tilde{C}_{p,q} \frac{\log n}{\sqrt{n}}$ and $\beta_n = \overline{C}_{p,q} n$, where $ \tilde{C}_{p,q} , \overline{C}_{p,q} \in (0, \infty)$ are suitably chosen constants only depending on $p$ and $q$. As shown in \cite[Section 5.3]{JPBerryEsseenlp}, there exist constants $ C_{i,p,q} \in (0, \infty)$, $i=1,2,3$, such that
\begin{equation*}
P_n \left (\frac{S_n^{(i)}}{\sqrt{n}} \right ) \leq \frac{C_{i,p,q}}{\sqrt{n}}, \quad i=1,2,3 .
\end{equation*}
For the quantity in \eqref{EqKolDistRZ}, by combining the previous estimate and \eqref{EqBerryEssenTaylorPoly}, we get 
\begin{equation*}
d_{Kol} \left( \sqrt{n} \big(  \mathcal{R}_{p,q}^{(n)} - m_{p,q} \Big), Z   \right) \leq \frac{\hat{C}_{p,q}}{\sqrt{n}}+ \frac{\epsilon_n }{\sqrt{2 \pi \sigma^2 }},
\end{equation*}
with $\hat{C}_{p,q} := C_1 + C_{1,p,q} + C_{2,p,q} + C_{3,p,q}$. Since $\epsilon_n = \tilde{C}_{p,q} \frac{\log n}{\sqrt{n}}$, we can find a constant $ C_{p,q} \in (0 , \infty)$ such that
\begin{equation}
\label{EqBEBoundConeMeas}
d_{Kol} \left( \sqrt{n} \big(  \mathcal{R}_{p,q}^{(n)} - m_{p,q} \Big), Z   \right) \leq C_{p,q} \frac{\log n}{\sqrt{n}}, \quad n \in \N ,
\end{equation}
as claimed. 
\par{}
Now we consider the case when $(X^{(n)}, Y^{(n)}) \sim \sigma^{(n)}_p \otimes \sigma^{(n)}_q$ and recall that $\mathcal{R}_{p,q}^{(n)} = \frac{\sum_{i=1}^n |X^{(n)}_i Y^{(n)}_i |}{||X^{(n)}||_p ||Y^{(n)}||_q}$. Further, let $(\tilde{X}^{(n)}, \tilde{Y}^{(n)}) \sim \mu^{(n)}_p \otimes \mu^{(n)}_q$ and define $\tilde{\mathcal{R}}_{p,q}^{(n)} := \frac{\sum_{i=1}^n |\tilde{X}^{(n)}_i \tilde{Y}^{(n)}_i |}{||\tilde{X}^{(n)}||_p ||\tilde{Y}^{(n)}||_q} $. We want to show that there exists a constant $C \in (0, \infty)$ such that
\begin{equation*}
d_{Kol} \left( \sqrt{n} \big(  \mathcal{R}_{p,q}^{(n)} - m_{p,q} \Big), \sqrt{n} \big(  \mathscr{\tilde{R}}_{p,q}^{(n)} - m_{p,q} \Big)   \right) \leq \frac{C}{\sqrt{n}}.
\end{equation*}
For a fixed $t \in \R$, we recall that $\mathbb{P}[ \mathcal{R}_{p,q}^{(n)} \geq t ] = \sigma^{(n)}_p \otimes \sigma^{(n)}_q (D_{n,t})$ as well as $ \mathbb{P}[ \mathscr{\tilde{R}}_{p,q}^{(n)} \geq t ] = \mu^{(n)}_p \otimes \mu^{(n)}_q (D_{n,t})$
with the set $D_{n,t} = \Big \{ (x,y) \in \mathbb{S}_p^{n-1} \times \mathbb{S}_q^{n-1 } \ : \ \frac{\sum_{i=1}^{n} |x_i y_i|}{||x||_p ||y||_q} \geq t \Big \}$. Thus, for all $n \in \N$, we have
\begin{align}
\label{EqKolDist1}
d_{Kol} \left( \sqrt{n} \big(  \mathcal{R}_{p,q}^{(n)} - m_{p,q} \Big), \sqrt{n} \big(  \mathscr{\tilde{R}}_{p,q}^{(n)} - m_{p,q} \Big)   \right)  & =  d_{Kol} \left(   \mathcal{R}_{p,q}^{(n)} , \mathscr{\tilde{R}}_{p,q}^{(n)}   \right)\\
\notag
& = \sup_{t \in \R} \big | \sigma^{(n)}_p \otimes \sigma^{(n)}_q (D_{n,t}) - \mu^{(n)}_p \otimes \mu^{(n)}_q(D_{n,t}) \big | \\
\notag
& \leq || \sigma^{(n)}_p \otimes \sigma^{(n)}_q - \mu^{(n)}_p \otimes \mu^{(n)}_q ||_{TV} \\
\label{EqKolDist2}
& \leq \frac{C}{\sqrt{n}}.
\end{align}

Equation \eqref{EqKolDist1} follows immediately from the definition of the Kolmogorov distance. The estimate in \eqref{EqKolDist2} follows from Proposition \ref{PropSurfMeasConeMeasClose}, since, for some $A = A_1 \times A_2$ with $A_1 \in \mathscr{B}( \mathbb{S}_p^{n-1})$ and $A_2 \in \mathscr{B}( \mathbb{S}_q^{n-1})$, we have
\begin{align*}
\left | \sigma_p^{(n)} \otimes \sigma_q^{(n)}(A) - \mu_p^{(n)} \otimes \mu_q^{(n)}(A) \right | 
& \leq  
\left | \sigma_p^{(n)}(A_1) -  \mu_p^{(n)}(A_1) \right | + 
\left | \sigma_q^{(n)}(A_2) -  \mu_q^{(n)}(A_2) \right | \\
& \leq \left | \left | \sigma_p^{(n)} - \mu_p^{(n)}  \right | \right |_{TV} + \left | \left | \sigma_q^{(n)} - \mu_q^{(n)} \right | \right |_{TV} \\
& \leq \frac{C}{\sqrt{n}}.
\end{align*}
By maximizing over all such $A$, we get $ || \sigma^{(n)}_p \otimes \sigma^{(n)}_q - \mu^{(n)}_p \otimes \mu^{(n)}_q ||_{TV} \leq \frac{C}{\sqrt{n}}$.
Now, for a $Z \sim \mathcal{N}(0,\sigma_{p,q}^2)$, we get the following estimate
\begin{align*}
d_{Kol} \left( \sqrt{n} \big(  \mathcal{R}_{p,q}^{(n)} - m_{p,q} \Big), Z   \right)  & \leq 
d_{Kol} \left( \sqrt{n} \big(  \mathscr{\tilde{R}}_{p,q} - m_{p,q} \Big), \sqrt{n} \big(  \mathcal{R}_{p,q}^{(n)} - m_{p,q} \Big)   \right) 
+ d_{Kol} \left( \sqrt{n} \Big( \mathscr{\tilde{R}}_{p,q} - m_{p,q} \Big), Z  \right) \\
& \leq \frac{C}{\sqrt{n}} + C_{p,q} \frac{\log n}{\sqrt{n}} \\
& \leq \left(C+ C_{p,q}\right) \frac{\log n}{\sqrt{n}} ,
\end{align*} 
where we used the bound established in \eqref{EqBEBoundConeMeas} and the first part of this proof.
\end{proof}

\subsection{Proof of Theorem \ref{ThmLDP}}

In the proofs of our main results, we frequently use the probabilistic representation of random variables distributed according to the cone measure $\mu_p^{(n)}$ (see Equation \eqref{EqProbRepSchechtmannZinn}). For the surface measure things are more delicate as we do not have such a representation. In order to establish large deviation and moderate deviation results for the surface measure $\sigma_p^{(n)}$, we will need the following exponential equivalence of the cone measure $\mu_p^{(n)}$ and the surface measure $\sigma_p^{(n)}$.
\begin{lem}
\label{LemExpEquivSurfConeMeas}
Let $ A \in \mathcal{B}(\R)$, $p,q \in (1,\infty)$ with $\frac{1}{p} + \frac{1}{q} = 1$ and define
\begin{equation*}
D_n:= \Big \{ (x,y) \in \mathbb{S}_p^{n-1} \times \mathbb{S}_q^{n-1} \ : \ \frac{\sum_{i=1}^n |x_i y_i |}{||x||_p ||y||_q} \in A \Big \}.
\end{equation*}
Then, it holds that
\begin{equation*}
\lim_{n \rightarrow \infty} \Big | \frac{1}{s_n} \log \sigma^{(n)}_p \otimes \sigma^{(n)}_q (D_n) - \frac{1}{s_n} \log \mu^{(n)}_p \otimes \mu^{(n)}_q (D_n) \Big| = 0,
\end{equation*}
where $(s_n)_{n \in \N}$ is a positive sequence with $\lim_{n \rightarrow \infty} \frac{\log n}{s_n}=0$.
\end{lem} 
\begin{proof}
Recall the Lebesgue-density $\frac{d \sigma^{(n)}_p}{d \mu^{(n)}_p}(x)=C_{n,p} \Big( \sum_{i=1}^{n} |x_i|^{2p-2} \Big)^{1/2} =: h_{n,p}(x)$, where $C_{n,p}\in(0,\infty)$ denotes the normalizing constant. By Lemma 2.2 in \cite{ArithGeoIneq}, there exists a constant $C \in (0, \infty)$ such that for all $x \in \mathbb{S}_p^{n-1}$, we have $ n^{-C} \leq h_{n,p}(x) \leq n^C$. Further, we can write $D_n = D_n^1 \times D_n^2$, where $D_n^1 \in \mathscr{B}(\mathbb{S}_p^{n-1})$ and $D_n^2 \in \mathscr{B}(\mathbb{S}_q^{n-1})$ (note that $D_n \in \mathscr{B}(\mathbb{S}_p^{n-1} \times \mathbb{S}_q^{n-1}) = \mathscr{B}(\mathbb{S}_p^{n-1}) \times \mathscr{B}(\mathbb{S}_q^{n-1}) $, where the latter holds, e.g., by Theorem D.4 in \cite{DZ2011}). We hence get the following estimate
\begin{align}
\begin{split}
\Big| \frac{1}{s_n} \log \sigma^{(n)}_p \otimes \sigma^{(n)}_q( D_n) -  \frac{1}{s_n} \log \mu^{(n)}_p \otimes \mu^{(n)}_q( D_n) \Big|  \leq &
\Big| \frac{1}{s_n} \log  \sigma^{(n)}_p ( D^1_n) -   \frac{1}{s_n}  \log \mu^{(n)}_p ( D^1_n) \Big| \\
\label{EqEstimateConeSurf}
 & + \Big| \frac{1}{s_n} \log \sigma^{(n)}_q ( D^2_n) -  \frac{1}{s_n} \log \mu^{(n)}_q ( D^2_n) \Big|.
\end{split}
\end{align}
We consider the first expression on the right-hand side in \eqref{EqEstimateConeSurf}, where we have
\begin{align*}
\Big| \frac{1}{s_n} \log \sigma^{(n)}_p ( D^1_n) -  \frac{1}{s_n} \log \mu^{(n)}_p ( D^1_n) \Big|  =\left | \frac{1}{s_n} \log \left( \frac{\sigma_p^{(n)}(D_n^1)}{\mu_p^{(n)}(D_n^1)} \right) \right |  \leq C \frac{\log \left ( n \right )}{s_n}    \stackrel{n \rightarrow \infty}{\longrightarrow} 0.
\end{align*}
Here, we used that $  n^{-C} \mu_p^{(n)}( D_n^1) \leq \sigma_p^{(n)} ( D_n^1) \leq n^C \mu_p^{(n)}( D_n^1) $. The second term in \eqref{EqEstimateConeSurf} can be treated analogously. 
\end{proof}

We will now present the proof of the large deviation principle.

\begin{proof}[Proof of Theorem \ref{ThmLDP}]
	First, assume that either, $(X^{(n)},Y^{(n)}) \sim \U( \mathbb{B}^n_p) \otimes  \U( \mathbb{B}^n_q)$ or $ (X^{(n)},Y^{(n)}) \sim \mu^{(n)}_p \otimes  \mu^{(n)}_q$. We use the probabilistic representation of $( \mathcal{R}_{p,q}^{(n)})_{n \in \N}$ and Cramér's theorem (see Proposition \ref{ThmCramer}) together with the contraction principle (see Lemma  \ref{LemContractionPrinciple}). By Lemma \ref{LemReprRpq}, we have that
	\begin{equation*}
		\mathcal{R}_{p,q}^{(n)} \stackrel{d}{=} \frac{\sum_{i=1}^{n} |\zeta_i \eta_i |}{\Big( \sum_{i=1}^{n} |\zeta_i|^p \Big)^{1/p} \Big( \sum_{i=1}^{n} |\eta_i|^q \Big)^{1/q}}, \quad n \in \N, 
	\end{equation*}
	where $(\zeta_i)_{i \in \N}$ is an iid sequence of $p$-generalized Gaussian distributed random variables and $(\eta_i)_{i \in \N}$ is an iid sequence of $q$-generalized Gaussian distributed random variables and both are independent. We consider the sequence $( \xi_n)_{n \in \N}$ with
	\begin{equation}
		\label{EqTrippleSeq}
		\xi_n := \frac{1}{n} \sum_{i=1}^n \big(  | \zeta_i \eta_i | , | \zeta_i |^p , | \eta_i |^q \big) \in \R^3, \quad n \in \N .
	\end{equation}
	The summands on the right-hand side in \eqref{EqTrippleSeq} are iid and thus, we want to apply Cramér's theorem (see Proposition \ref{ThmCramer}) in order to establish an LDP for $( \xi_n)_{n \in \N}$. We do this by showing that the cumulant generating function $\Lambda: \R^3 \rightarrow [0, \infty] $ with
	\begin{equation}
		\label{EqCumGenFLdp}
		\Lambda(r,s,t ) := \log \E \big[ e^{ r |\zeta_1 \eta_1 | +s | \zeta_1 |^p + t |\eta_1 |^q}  \big],
	\end{equation} 
	is finite in some ball around $0 \in \R^3$. 
	Using the density in \eqref{EqDensitypGenGaussian}, we can write \eqref{EqCumGenFLdp} as
	\begin{equation*}
		\Lambda(r,s,t)=  \log\int_{\R^2}  e^{ r |x y | +s | x|^p + t |y |^q - \frac{1}{p} |x|^p - \frac{1}{q} |y|^q} c_{p,q} dx \,dy, 
	\end{equation*}
	where $c_{p,q} = \frac{1}{2 p^{1/p} \Gamma( 1 + \frac{1}{p})}  \frac{1}{2 q^{1/q} \Gamma( 1 + \frac{1}{q})}$. Then, $ \Lambda(r,s,t)$ is finite if and only if
	\begin{equation*}
		c_{p,q} \int_{\R} \int_{\R}   e^{ r |x y|  + \big(  s   - \frac{1}{p} \big)  |x|^p + \big( t - \frac{1}{q} \big) |y| ^q } dx\, dy < \infty. 
	\end{equation*}
	Let us fix a point $(r,s,t) \in \R^3 $ with $|r| < \epsilon$, $|s| < \epsilon$ and $|t| < \epsilon$, where we choose $ \epsilon:= \frac{1}{2(1+\max(q,p))}$. Then, using that $|xy| \leq \frac{1}{p}|x|^p + \frac{1}{q} |y|^q$ for all $x,y \in \R$, we get the following estimate
	\begin{align*}
		c_{p,q} \int_{\R} \int_{\R}   e^{ r |x y|  + \big(  s   - \frac{1}{p} \big)  |x|^p + \big( t - \frac{1}{q} \big) |y| ^q } dx\, dy  & 
		\leq c_{p,q} \int_{\R} \int_{\R} e^{\Big( s- \frac{1}{p} + \frac{r}{p} \Big) |x|^p} e^{\Big( t- \frac{1}{q} + \frac{r}{q} \Big) |y|^q}dx \,dy < \infty.
	\end{align*}
	The integral on the right-hand side is finite, since $ \Big( s - \frac{1}{p} + \frac{r}{p} \Big) < 0$ and $ \Big( t - \frac{1}{q} + \frac{r}{q} \Big) < 0 $ by our choice of $\epsilon $.
	So, we are able to find an open neighborhood $U \subseteq \R^3$ around zero such that $\Lambda(r,s,t) < \infty$ for all $( r,s,t ) \in U$.
	Now we can apply Cramér's theorem to the sequence $( \xi_n)_{n \in \N}$ defined in \eqref{EqTrippleSeq}. It follows that $( \xi_n)_{n \in \N}$ satisfies an LDP in $\R^3$ at speed $n$ with GRF $\Lambda^{*}: \R^3 \rightarrow [0, \infty]$, where
	\begin{equation*}
		\Lambda^{*}( u,v,w) := \sup_{ (r,s,t) \in \R^3 } \big[ ru + tv +sw - \Lambda(r,s,t) \big].
	\end{equation*} 
	Consider the continuous mapping $F : \R \times (0, \infty)^2 \rightarrow [0, \infty]$ with
	\begin{equation*}
		F( u, v,w) := \frac{u}{v^{1/p} w^{1/q}}.
	\end{equation*}
	We see that
	\begin{align*}
		F( \xi_n ) & = \frac{ \frac{1}{n} \sum_{i=1}^n |\zeta_i \eta_i | }{ \big(  \frac{1}{n} \sum_{i=1}^n | \zeta_i |^p \big)^{1/p}  \big(  \frac{1}{n} \sum_{i=1}^n | \eta_i |^q \big)^{1/q} } \\
		& \stackrel{d}{=} \mathcal{R}_{p,q}^{(n)}, \quad n \in \N,
	\end{align*}
	where we used that $ \frac{1}{p}+ \frac{1}{q} =1$. Hence, the contraction principle (see Lemma \ref{LemContractionPrinciple}) applied to the sequence $ ( F( \xi_n ))_{n \in \N}$ gives us an LDP for $ ( \mathcal{R}_{p,q}^{(n)})_{n \in \N}$ at speed $n$ with GRF $ \I : \R \rightarrow [0, \infty]$, where
	\begin{equation*}
		\I( x) = \begin{cases}
			\inf \Big \{  \Lambda^{*}( u,v,w) \ : \  x = \frac{u}{v^{1/p} w^{1/q}} \Big \}, & \text{if } x > 0 \\
			+ \infty & \text{ else}.
		\end{cases}
	\end{equation*}
	Now we assume that $ (X^{(n)}, Y^{(n)}) \sim \sigma^{(n)}_p \otimes \sigma^{(n)}_q , n \in \N$. For a closed set $A \subseteq \R $, by Lemma \ref{LemExpEquivSurfConeMeas}, we get
	\begin{equation*}
		\limsup_{ n \rightarrow \infty } \frac{1}{n} \log \mathbb{P} \big [ \mathcal{R}_{p,q}^{(n)} \in A  \big] = \limsup_{ n \rightarrow \infty } \frac{1}{n} \log \sigma^{(n)}_p \otimes \sigma^{(n)}_q  (D_n) = 
		\limsup_{ n \rightarrow \infty } \frac{1}{n} \log \mu^{(n)}_p \otimes \mu^{(n)}_q  (D_n) \leq - \inf_{x \in A} \mathbb{I}(x).
	\end{equation*}
	$D_n= \Big \{ (x,y) \in \mathbb{S}_p^{n-1} \times \mathbb{S}_q^{n-1} \ : \ \frac{\sum_{i=1}^n |x_i y_i |}{||x||_p ||y||_q} \in A \Big \}$ is the set from Lemma \ref{LemExpEquivSurfConeMeas} and we have used the large deviation upper bound which holds for $\mu^{(n)}_p \otimes \mu^{(n)}_q$. This proves the large deviation upper bound for $\sigma^{(n)}_p \otimes \sigma^{(n)}_q$. The lower bound can be shown analogously. 
\end{proof}

\subsection{Proof of Theorem \ref{ThmMDP}} 
\label{SectionProofMDP}

We now present the proof of the MDP.

\begin{proof}[Proof of Theorem \ref{ThmMDP}]
First, we assume that either, $(X^{(n)}, Y^{(n)}) \sim \U(\mathbb{B}_p^{n}) \otimes \U(\mathbb{B}_q^{n})$ or $(X^{(n)}, Y^{(n)}) \sim \mu^{(n)}_p \otimes \mu^{(n)}_q$. We work with the sequence 
\begin{equation*}
\frac{\sqrt{n}}{b_n} \left(  \mathcal{R}_{p,q}^{(n)} - m_{p,q}  \right), \quad n \in \N,
\end{equation*}
where $\mathcal{R}_{p,q}^{(n)}$ is the quantity from Equation \eqref{EqRepRationRpq} and $m_{p,q} = p^{1/p} \frac{\Gamma( \frac{2}{p})}{\Gamma( \frac{1}{p})}   q^{1/q} \frac{\Gamma( \frac{2}{q})}{\Gamma( \frac{1}{q})} $. Using the Taylor expansion of $\mathcal{R}_{p,q}^{(n)}$ given in Lemma \ref{LemTaylorCLT}, we obtain
\begin{equation}
\label{EqTaylorExpRMDP}
\frac{\sqrt{n}}{b_n} \left(  \mathcal{R}_{p,q}^{(n)} - m_{p,q}  \right) \stackrel{d}{=} \frac{1}{b_n} S_n^{(1)} - \frac{m_{p,q}}{b_n p} S_n^{(2) } - \frac{m_{p,q}}{b_n q} S_n^{(3) } + \frac{\sqrt{n}}{b_n} R \left(  \frac{S_n^{(1)}}{\sqrt{n}} , \frac{S_n^{(2)}}{\sqrt{n}} , \frac{S_n^{(3)}}{\sqrt{n}}\right),
\end{equation}
where $ S_n^{(1)} := \frac{1}{\sqrt{n}} \sum_{i=1}^n ( | \zeta_i \eta_i | - m_{p , q}) $, $S_n^{(2)} := \frac{1}{\sqrt{n}} \sum_{i=1}^n ( | \zeta_i |^p -  1) $ and $S_n^{(3)} := \frac{1}{\sqrt{n}} \sum_{i=1}^n ( | \eta_i |^q -  1)$ with iid $(\zeta_i , \eta_i ) \sim \gamma_p \otimes \gamma_q$, $i \in \N$. First, we show that the quantity in \eqref{EqTaylorExpRMDP} and 
\begin{equation}
\label{EqTaylorExpMDP}
Y_n := \frac{1}{b_n} S_n^{(1)} - \frac{m_{p,q}}{b_n p} S_n^{(2) } - \frac{m_{p,q}}{b_n q} S_n^{(3) }, \quad n \in \N 
\end{equation}
are exponentially equivalent on the scale $ (b_n^2)_{n \in \N}$. To that end, for $\epsilon >0$, we consider
\begin{align*}
\frac{1}{b_n^2} \log \mathbb{P} \left[ \left | \frac{\sqrt{n}}{b_n} \left(  \mathcal{R}_{p,q}^{(n)} - m_{p,q}  \right) -  Y_n   \right | > \epsilon   \right] & =  
\frac{1}{b_n^2} \log \mathbb{P} \left[ \left | \frac{\sqrt{n}}{b_n} R \left(  \frac{S_n^{(1)}}{\sqrt{n}} , \frac{S_n^{(2)}}{\sqrt{n}} , \frac{S_n^{(3)}}{\sqrt{n}}\right)   \right | > \epsilon    \right].
\end{align*}
Here, we take a closer look at the probability in the logarithm on the right-hand side. For sufficiently large $n \in \N$, we get
\begin{align*}
\mathbb{P} \left[ \left | \frac{\sqrt{n}}{b_n} R \left(  \frac{S_n^{(1)}}{\sqrt{n}} , \frac{S_n^{(2)}}{\sqrt{n}} , \frac{S_n^{(3)}}{\sqrt{n}}\right)   \right | > \epsilon \right] & \leq 
\mathbb{P} \left[ \left | \left |  \left(  \frac{S_n^{(1)}}{\sqrt{n}} , \frac{S_n^{(2)}}{\sqrt{n}} , \frac{S_n^{(3)}}{\sqrt{n}} \right)  \right | \right |_2^2  \geq \delta   \right] 
+ \mathbb{P} \left[ \left | \left |  \left(  \frac{S_n^{(1)}}{\sqrt{n}} , \frac{S_n^{(2)}}{\sqrt{n}} , \frac{S_n^{(3)}}{\sqrt{n}} \right)  \right |  \right |_2^2 > \frac{\epsilon}{M} \frac{b_n}{\sqrt{n}}    \right] \\
& \leq 2 \mathbb{P} \left[ \left | \left |  \left(  \frac{S_n^{(1)}}{\sqrt{n}} , \frac{S_n^{(2)}}{\sqrt{n}} , \frac{S_n^{(3)}}{\sqrt{n}} \right)  \right |  \right |_2^2 > \frac{\epsilon}{M} \frac{b_n}{\sqrt{n}}    \right] \\
& = 2 \mathbb{P} \left[ \left | \left |  \left(  \frac{S_n^{(1)}}{b_n} , \frac{S_n^{(2)}}{b_n} , \frac{S_n^{(3)}}{b_n} \right)  \right |  \right |_2^2 > \frac{\epsilon}{M} \frac{\sqrt{n}}{b_n}    \right], 
\end{align*}
where we used the properties of the function $R$ from Lemma \ref{LemTaylorCLT}. For a value $T \in (0, \infty)$, we have $ \frac{\epsilon}{M} \frac{\sqrt{n}}{b_n} > T$ for all sufficiently large $n \in \N$, since $ \frac{b_n}{\sqrt{n}}$ tends to zero as $n \rightarrow \infty$. The sequence $ \left(  \frac{S_n^{(1)}}{b_n} , \frac{S_n^{(2)}}{b_n} , \frac{S_n^{(3)}}{b_n}  \right)_{n \in \N}$ satisfies an MDP at speed $( b_n^2)_{n \in \N}$ by Proposition \ref{ThmCramerMDP} with GRF $ \mathbb{J} : \R^3 \rightarrow [0, \infty)$, where
\begin{equation}
\label{EqGRFMDP}
\mathbb{J}(x) := \frac{1}{2}\langle x, \mathbf{C_{p,q}}^{-1} x \rangle , \quad x \in \R^3. 
\end{equation}
$ \mathbf{C_{p,q}} \in \R^{3x3}$ is the positive definite covariance matrix of the vector $ \left(  | \zeta_1 \eta_1 | , | \zeta_1 |^p, | \eta_1 |^q \right) \in \R^{3}$ (see the proof of Theorem \ref{ThmCLT}). Thus, by combining the upper bound of the MDP and the contraction principle (see Lemma \ref{LemContractionPrinciple}), we get
\begin{align*}
\limsup_{ n \rightarrow \infty } \frac{1}{b_n^2} \log \mathbb{P} \left[ \left | \frac{\sqrt{n}}{b_n} R \left(  \frac{S_n^{(1)}}{\sqrt{n}} , \frac{S_n^{(2)}}{\sqrt{n}} , \frac{S_n^{(3)}}{\sqrt{n}}\right)   \right | > \epsilon    \right] & \leq 
\limsup_{ n \rightarrow \infty } \frac{1}{b_n^2} \log
\mathbb{P} \left[ \left | \left |  \left(  \frac{S_n^{(1)}}{b_n} , \frac{S_n^{(2)}}{b_n} , \frac{S_n^{(3)}}{b_n} \right)  \right |  \right |_2^2 > T    \right] \\
& \leq - \inf \left \{ x^t \mathbf{C_{p,q}}^{-1} x \ : \ || x ||_2^2 > T   \right \}  \\
& \leq - T \lambda_{min}( \mathbf{C_{p,q}}^{-1}),
\end{align*} 
where $ \lambda_{min}( \mathbf{C_{p,q}}^{-1}) \in (0, \infty)  $ denotes the smallest eigenvalue of $ \mathbf{C_{p,q}}^{-1}$. Since the previous bound holds for any $T \in (0, \infty)$, we have
\begin{equation*}
\limsup_{ n \rightarrow \infty } \frac{1}{b_n^2} \log \mathbb{P} \left[ \left | \frac{\sqrt{n}}{b_n} R \left(  \frac{S_n^{(1)}}{\sqrt{n}} , \frac{S_n^{(2)}}{\sqrt{n}} , \frac{S_n^{(3)}}{\sqrt{n}}\right)   \right | > \epsilon    \right] = - \infty, 
\end{equation*}
showing the exponential equivalence at scale $(b_n^2)_{n \in \N }$ as desired. Given that, we need to prove an MDP at rate $(b_n^2)_{n \in \N }$ for the sequence $(Y_n)_{n \in \N}$ defined in Equation \eqref{EqTaylorExpMDP}. For $n \in \N$, we have $Y_n = \left \langle \left(  \frac{S_n^{(1)}}{b_n} , \frac{S_n^{(2)}}{b_n} , \frac{S_n^{(3)}}{b_n}\right) , d_{p,q} \right  \rangle $, where $d_{p,q} = \left( 1, - \frac{m_{p,q}}{p}, - \frac{m_{p,q}}{q} \right)$. We recall that $\left(  \frac{S_n^{(1)}}{b_n} , \frac{S_n^{(2)}}{b_n} , \frac{S_n^{(3)}}{b_n} \right)_{n \in \N } $ satisfies an MDP in $\R^3 $ at speed $ (b_n^2)_{n \in \N}$ with GRF $ \mathbb{J} : \R^3 \rightarrow [0, \infty)$. Now, we can apply the contraction principle (see Lemma \ref{LemContractionPrinciple}) in order to establish an MDP for the sequence $( Y_n)_{n \in \N}$ in $\R$ at speed $( b_n^2)_{n \in \N}$ with GRF $ \I : \R \rightarrow [0, \infty ]$, where
\begin{equation}
\label{EqGRFMDPSeqY}
\I( t) := \inf \left \{ \frac{1}{2}\langle x , \mathbf{C_{p,q}}^{-1} x \rangle  \ : \ \langle d_{p,q} , x \rangle = t    \right \}.
\end{equation}
We are able to give a closed form of $ \I$ by using the Lagrange method for optimization. For this, we fix $ t \in \R$ and consider the Lagrange function $ L: \R^4 \rightarrow \R$ given by
\begin{equation*}
L( x, \lambda ): =\frac{1}{2} \langle x , \mathbf{C_{p,q}}^{-1} x \rangle + \lambda ( t - \langle d_{p,q} , x \rangle ), \quad x \in \R^3 , \quad \lambda \in \R .
\end{equation*}
The directional derivatives are
\begin{equation*}
\left (  \frac{\partial  L( x , \lambda)}{ \partial x_1 }, \frac{\partial  L( x , \lambda)}{ \partial x_2 }  , \frac{\partial  L( x , \lambda)}{ \partial x_3 } \right)   = \mathbf{C_{p,q}}^{-1} x - \lambda d_{p,q} =0 \in \R^3 \quad \text{and} \quad  \frac{\partial }{ \partial \lambda } L( x , \lambda ) = t - \langle d_{p,q} ,x \rangle = 0.
\end{equation*}
This system of equations can be solved elementary and for the solution $x = x(t)$, we get
\begin{equation*}
x(t)= \frac{ \mathbf{C_{p,q}} d_{p,q} }{ \langle d_{p,q} , \mathbf{C_{p,q}} d_{p,q} \rangle } t,
\end{equation*}
where we mention that $d_{p,q} \neq (0,0,0)$ and hence $ \langle d_{p,q} , \mathbf{C_{p,q}} d_{p,q} \rangle \in (0, \infty )$. Finally, our GRF $\I : \R \rightarrow [0, \infty]$ is given as
\begin{equation*}
\I( t) = \frac{t^2}{2 \langle d_{p,q} , \mathbf{C_{p,q}} d_{p,q} \rangle}, \quad t \in \R, 
\end{equation*}
where the quantity $ \sigma_{p,q}^2 = \langle d_{p,q} , \mathbf{C_{p,q}} d_{p,q} \rangle \in (0, \infty ) $ is as claimed in Theorem \ref{ThmMDP}.
\par{}
Now we consider the case when $(X^{(n)}, Y^{(n)}) \sim \sigma^{(n)}_p \otimes \sigma^{(n)}_q$, $n \in \N$. We fix a closed set $A \subseteq \R$ and denote $ D_n = \Big \{  (x,y) \in \mathbb{S}_p^{n-1} \times \mathbb{S}_q^{n-1} \ : \ \frac{ \sum_{i=1}^n |x_i y_i |}{|| x ||_p ||y||_q} \in A \Big \}$. We then get
\begin{equation*}
\limsup_{n \rightarrow \infty} \frac{1}{b_n^2} \log \mathbb{P} \left[ \mathcal{R}_{p,q}^{(n)} \in A \right] = \limsup_{n \rightarrow \infty} \frac{1}{b_n^2} \log \sigma^{(n)}_p \otimes \sigma^{(n)}_q(D_n)  
= 
\limsup_{n \rightarrow \infty} \frac{1}{b_n^2} \log \mu^{(n)}_p \otimes \mu^{(n)}_q(D_n)  \leq - \inf_{t \in A} \mathbb{I}(t),
\end{equation*} 
where we have used Lemma \ref{LemExpEquivSurfConeMeas} and the assumption $\lim_{n \rightarrow \infty} \frac{b_n}{\sqrt{\log n}} = \infty$. This establishes the upper bound of the MDP. The lower bound can be shown analogously.
\end{proof}

\subsection*{Acknowledgement}
LF and JP are supported by the Austrian Science Fund (FWF) Project P32405 \textit{Asymptotic geometric analysis and applications} and the FWF Project F5513-N26, which is a part of the Special Research Program \textit{Quasi-Monte Carlo Methods: Theory and Applications}. LF is also supported by the FWF Project P 35322 \textit{Zufall und Determinismus in Analysis und Zahlentheorie}.
This work is part of the Ph.D. thesis of LF written under supervision of JP.

\bibliographystyle{plain}
\bibliography{HoelderIneq_bib.bib}

\bigskip
\bigskip

\medskip

\small

\noindent \textsc{Lorenz Fr\"uhwirth:} Institut für Analysis und Zahlentheorie,
Graz University of Technology, Kopernikusgasse 24/II, 8010 Graz, Austria

\noindent
\textit{E-mail:} \texttt{lorenz.fruehwirth@tugraz.at}

\medskip

\small

\noindent \textsc{Joscha Prochno:} Faculty of Computer Science and Mathematics, University of Passau, Dr.-Hans-Kapfinger-Stra{\ss}e 30, 94032 Passau, Germany. 

\noindent
\textit{E-mail:} \texttt{joscha.prochno@uni-passau.de}

\end{document}